\documentclass[paper=a4, fontsize=12pt]{scrartcl}

\usepackage[scaled]{helvet}
\usepackage[T1]{fontenc}

\usepackage[english]{babel}															
\usepackage[protrusion=true,expansion=true]{microtype}	
\usepackage{amsmath,amsfonts,amsthm} 
\usepackage{bm} 
\usepackage{dsfont}
\usepackage[pdftex]{graphicx}	
\usepackage{enumitem}

\usepackage{url}
\usepackage{colortbl}
\usepackage{booktabs}
\usepackage{tabularx} 
\usepackage{float} 
\usepackage{multirow}%
\usepackage[dvipsnames]{xcolor}
\usepackage{fullpage}
\usepackage{natbib}
\usepackage{subcaption}  
\usepackage{graphicx}
\usepackage{tabularx} 
\usepackage{float} 
\usepackage{caption}

\usepackage{xr-hyper} 

\usepackage[colorlinks=true,citecolor=blue,urlcolor=blue]{hyperref}
\usepackage[capitalise]{cleveref}

\newtheorem{theorem}{Theorem}

\newtheorem{corollary}[theorem]{Corollary}

\newtheorem{lemma}[theorem]{Lemma}

\definecolor{offwhite}{RGB}{255,250,240}
\definecolor{gray}{RGB}{155,155,155}
\definecolor{foreground}{RGB}{80,80,80}
\definecolor{background}{RGB}{255,255,255}
\definecolor{title}{RGB}{89,132,212}
\definecolor{subtitle}{RGB}{255,199,0}
\definecolor{hilit}{RGB}{248,117,79}
\definecolor{vhilit}{RGB}{255,111,207}
\definecolor{lolit}{RGB}{200,200,200}
\definecolor{lit}{RGB}{255,199,0}
\definecolor{mdlit}{RGB}{71,106,170}
\definecolor{link}{RGB}{248,117,79}
\definecolor{subcol}{RGB}{207,216,233}

\definecolor{darkraspberry}{rgb}{0.53, 0.15, 0.34}

\usepackage{sectsty}

\usepackage{fancyhdr}
\usepackage{chngcntr}
\counterwithout{table}{section}
\counterwithout{figure}{section}

\title{
		A uniform relative deviation inequality for VC-subgraph classes
}

\date{}

\author{\normalfont François Portier\footnote{Affiliation: CREST-ENSAI, University of Rennes 1, email: francois.portier@gmail.com, date: November 27th,  2026.}}

\begin{document}

\maketitle

\begin{abstract}
We establish a new Bernstein-type deviation inequality for classes of functions whose complexity is characterized through
subgraphs. The inequality is non-asymptotic, involves explicit constants, and features a relative normalization by the probability level. Applied to kernel density estimation, it produces a 
location and bandwidth-adaptive error bound between the estimator and the smoothed density, holding
simultaneously over all points on the real line and all positive bandwidths. The
proof is elementary, combining a new symmetrization principle, which incorporates the
relative normalization, with the maximal sub-Gaussian inequality, and requires neither
concentration nor entropy-integral arguments. When specialized to classes of sets, our technique reduces the factor $4\,\mathbb S_{\mathcal A}(2n)$ in the classical inequality for $P-P_n$ (normalized by $\sqrt P$)
to $(3/2)\,\mathbb S_{\mathcal A}(2n)$, and yields, with the factor $\mathbb S_{\mathcal A}(2n)$,
the companion inequality for $P_n-P$ (normalized by $\sqrt{P_n}$).

\noindent \textbf{Keyword:} Vapnik–Chervonenkis inequality; Uniform relative deviation inequality; Kernel density estimation; Adaptive confidence bands.

\end{abstract}


\section{Introduction}
Vapnik--Chervonenkis (VC) inequalities \citep{vapnik2015uniform,vapnik2013nature} provide uniform deviation bounds on the estimation error over classes of sets. They underlie many fundamental results in probability and statistics, including the uniform central limit theorem \citep{pollard1982central,alexander1987central}, generalization bounds in learning theory \citep[Theorem~7]{bousquet2003introduction}, and are a cornerstone of empirical process theory \citep{wellner1996}. They are also used to analyze the convergence of cluster trees \cite[Theorem~15]{chaudhuri2010rates}, to control the size of nearest neighbor radii \cite[Lemma~3]{xue2018achieving} and \cite[Lemma 3]{portier2025nearest}, and to evaluate the performance of neural network training \cite[Theorem~6]{bartlett1998sample} (see also \cite{anthony2009neural}).

Let $(Z_1,\dots,Z_n)$ be independent and identically distributed random variables with distribution $P$ on a measurable space $S$, and let $\mathcal{A}$ be a class of measurable subsets of $S$. Let $\mathbb{S}_{\mathcal{A}}(n)$ denote the shatter coefficient of $\mathcal A$, defined as the maximal number, over all $(z_1,\ldots,z_n)\in S^n$, of vectors in $\{0,1\}^n$ of the form $(\mathds{1}_A(z_1),\ldots,\mathds{1}_A(z_n))$ with $A\in\mathcal{A}$. For any $A\in\mathcal{A}$, the empirical estimator of $P(A)$ is defined as $P_n(A) = n^{-1}\sum_{i=1}^n \mathds{1}_A(Z_i)$. The strongest version of the VC inequalities involves relative deviations and reads as follows: for all $t>0$,
$$\mathbb{P}\left(\exists A\in \mathcal A\, :\,  P(A)-P_n(A)  > t\sqrt{P(A)}\right) \leq 4\,\mathbb{S}_{\mathcal{A}}(2n)\exp\left(-\frac{nt^2} 4\right).$$
This inequality is, to the best of our knowledge, the sharpest available; it was obtained by \cite{anthony1993result}. Its proof combines a symmetrization of the tail probability, tailored to the relative deviation, with the standard sub-Gaussian property of the Rademacher process, working conditionally on $Z_1,\ldots,Z_n$. The symmetrization step is responsible for the constant $4$, as it involves the probability that a binomial random variable exceeds its expectation \citep{greenberg2014tight}. As pointed out in \cite{cortes2019relative}, the reverse inequality, in which the roles of $P$ and $P_n$ are exchanged, cannot be obtained by simply adapting this proof. The version derived in \cite{cortes2019relative} controls the quantity $({P_n(A)-P(A)})/\sqrt{P_n(A)+\tau}$ for some $\tau>0$, and is therefore weaker, owing to the additional shift $\tau$ in the normalization. Of course, weaker versions can be obtained without the relative deviation as for instance in Theorem 12.4 in \cite{devroye2013probabilistic}.

The first result of this work is the following: for all $t>0$,
$$\mathbb{P}\left(\exists A\in\mathcal{A} \,:\,  P(A)-P_n(A) > t\sqrt{\frac{P(A)}{n}}\right) \leq \frac 3 2\,\mathbb{S}_{\mathcal{A}}(2n)\exp\left(-\frac{t^2} 4\right),$$
and
$$\mathbb{P}\left( \exists A\in\mathcal{A} \,:\,  P_n(A)-P(A)  > t \sqrt{\frac{P_n(A)}{n}} \right) \leq \mathbb{S}_{\mathcal{A}}(2n)\exp\left(-\frac{t^2} 4\right).$$
The proof is elementary: neither concentration nor entropy-integral arguments are needed, only a \textit{symmetrization principle} combined with the \textit{maximal sub-Gaussian inequality}. The symmetrization principle relies on Jensen's inequality applied to suitably chosen convex functions. It extends those of \cite[Exercise~11.5]{boucheron2013concentration} and \cite[Lemma~2.3.6]{wellner1996}, the key innovation being that it incorporates the relative normalization by $\sqrt{P(A)}$ or $\sqrt{P_n(A)}$. In contrast with the proofs of \cite{anthony1993result} and \cite{cortes2019relative}, the binomial event mentioned above is no longer involved, which is how the constant $4$ is improved to $3/2$ for $P-P_n$, and how the companion inequality for $P_n-P$ is obtained with constant $1$.

The previous inequalities apply to probabilities of sets $A \in \mathcal{A}$, providing control over the discrepancy between $P(A)$ and $P_n(A)$. They do not extend directly to real-valued functions, that is, to the difference $P_n(f)-P(f)$, where $P_n(f) = \int f \, \mathrm{d}P_n$, $P(f) = \int f \, \mathrm{d}P$, and $f$ belongs to a class of real-valued functions defined on $S$. The main obstacle is simply that the shatter coefficient is a combinatorial concept defined for sets rather than functions. As pointed out in \cite{vapnik2013nature} (see the bibliographical notes of Chapter 5), there exist multiple ways to make this generalization, and a substantial body of work has been devoted to extending VC inequalities to real-valued function classes (see for instance \cite{wellner1996,boucheron2013concentration,devroye2013probabilistic}).

The second result of this work is a new deviation inequality valid uniformly over a class of functions. For a countable class $\mathcal F$ with subgraphs $\mathrm{sg}(\mathcal{F}) = \{\mathrm{sg}(f) \,:\,f\in \mathcal F\} $, where $\mathrm{sg}(f)=\{(x,y)\in S\times[0,1]:\ y<f(x)\}$, we establish that, for any $n\geq 1$, $t>0$ and $\gamma>0$ such that $t\sqrt{n\gamma} \geq 16/9$,
$$\mathbb{P}\left(\exists f\in\mathcal{F}\,:\, P_n(f)-P(f)  >\gamma+ t\sqrt{\frac{P_n(f)}{n}} \right) \leq \mathbb{S}_{\mathrm{sg}(\mathcal{F})}(2n)\exp\left(-\frac{t^2} 4\right),$$
and
$$\mathbb{P}\left(\exists f\in\mathcal{F}\,:\, P(f)-P_n(f) >\gamma + t \sqrt{\frac{P(f)}{n}} \right) \leq \mathbb{S}_{\mathrm{sg}(\mathcal{F})}(2n)\exp\left(-\frac{t^2} 4\right).$$
Choosing $\gamma = t^2/n$ and $t\ge 4/3$, the first inequality is similar to a Bernstein-type inequality with two terms in the upper bound: one scaling as $\sqrt{P_n(f) /n }$, for the first  and the other as $1/n$. A related inequality is obtained in \cite{bartlett1999inequality} under an $L_\infty$-covering number assumption on $\mathcal{F}$ (see also \cite{pollard1995uniform} for another such inequality). 
Let us also mention \cite{MaurerPontil2009} and \cite{audibert2009exploration}, where Bernstein-type inequalities featuring a data-dependent variance term are obtained; these are related to our bounds involving the empirical quantity $P_n(f)$. For the previous references, while their convergence rate is optimal, the covering-number requirement is restrictive, as the $L_\infty$-metric is used. Our inequality provides the same rate, but no $L_\infty$-covering numbers are needed: only the VC-subgraph property is required. This is particularly useful for localized classes of functions, for which the localization bandwidth does not affect the VC-subgraph complexity, as illustrated in Section \ref{sec:applications}. 

We consider an application to kernel density estimation, deriving a non-asymptotic uniform bound with explicit constants valid for any sample size. The result is a high-probability error bound between the estimator and the smoothed density $f_h$, holding simultaneously over all points on the real line. Unlike chaining-based uniform bounds \citep{nolan+p:1987,einmahl2000,gine+g:02,gine2009,portier2025nearest,baraud2016bounding}, the bound adapts to the local density level, yielding tighter guarantees in low-density regions. This location-adaptive property provides a natural foundation for data-driven confidence bands for the smoothed density, and opens the way for future construction of confidence bands for the true density, in the spirit of the bands obtained in \cite{juditsky2003nonparametric,robins2006adaptive,wasserman_2008,gine_nivkl_band}.

The outline is as follows. Section \ref{sec:main_results1} provides a refinement of the constants obtained in \cite{anthony1993result}. Section \ref{sec:main_results2} considers the case of real-valued function classes through their subgraphs and Section \ref{sec:applications} investigates an application to kernel density estimation. All the proofs are given in Section~\ref{sec:proofs}.

\section{A VC inequality for sets}\label{sec:main_results1}

Let $P$ be a probability measure on a measurable set $S$. Let $\mathcal{A}$ be a class of measurable subsets of $S$.
The \textit{shatter coefficient} of $\mathcal A$ is defined, for any $n\geq 1$, as $$\mathbb S_{\mathcal A}(n) : = \max_{(z_1,\ldots, z_n)\in S^n}\left|\{ (\mathrm 1 _A(z_1),\ldots ,\mathrm 1 _A(z_n)) \,:\, A\in \mathcal A\}\right|. $$
We shall assume throughout that $\mathcal A$ is countable for measurability reasons. 
The next theorem reduces the factor $4\,\mathbb S_{\mathcal A}(2n)$ of \cite{anthony1993result} to $(3/2)\,\mathbb S_{\mathcal A}(2n)$ for $P-P_n$
(second inequality, see below). The first inequality, for $P_n-P$ normalized by $\sqrt{P_n}$, appears to be new, as pointed out in \cite{cortes2019relative}.

\begin{theorem}\label{theorem:th1}
Let $(Z_1,\dots,Z_n)$ be independent and identically distributed random variables with distribution $P$ on $S$. Let $\mathcal{A}$ be a countable class of measurable subsets of $S$.  We have for all $t>0$,
$$ \mathbb P \left(  \exists A\in \mathcal A\,:\;    {P_n(A) -P (A) }  > {t} \sqrt{ \frac{ P_n (A )   } {n} } \right) \leq \mathbb S_{\mathcal A}(2n)\exp\left(-\frac{t^2} 4\right).$$
and
$$ \mathbb P \left(   \exists A\in \mathcal A\,:\;  {P(A) -P_n (A) } > t{\sqrt{ \frac{ P (A )}{n}   }}  \right) \leq\frac 3 2
\mathbb S_{\mathcal A}(2n)    \exp\left( - \frac{t^2 }{ 4}    \right) .$$

\end{theorem}

The asymmetry comes from convexity: Jensen's inequality applies to all sets for the first inequality (see the proof), but for the second only to those with $P_n(A)\ge1/n$ (and $t\ge7\sqrt2/9$); the sets with $P_n(A)=0$ are treated separately, which produces an additional term. In fact, we show in the proof, for every $t>0$,
\begin{align*}
 \mathbb P \left(  \sup_{A\in \mathcal A,\,P_n(A)  = 0} n   P(A)    > t^2 \right)\le \mathbb S_{\mathcal A}(2n)\,2^{-t^2}.
\end{align*}
Note that a similar but weaker result is obtained in \cite{blumer1989learnability} (see also the proof of Theorem 12.7 in \cite{devroye2013probabilistic}). Our proof differs from that of the previous references and relies mostly on Jensen's inequality and symmetrization (as in the other steps of the proof). An asymmetry between the two tails is also present in the multiplicative Chernoff bound \citep{hagerup1990guided}, which can be seen as a non-uniform version of Theorem~\ref{theorem:th1}. Whether the constants $3/2$ and $1$ appearing in front of $\mathbb S_{\mathcal A}(2n)$ in each bound are optimal remains an open question. The above inequalities can be inverted to yield bounds featuring either the empirical quantity $P_n$ or its population counterpart $P$, as follows.

\begin{corollary}\label{theorem:cor1}
In the setting of Theorem \ref{theorem:th1}, let $\delta\in(0,1)$. We have, with probability at least $1-\delta$, for all $A\in \mathcal A$,
\begin{align*}
&n(P_n(A)- P (A )) \\
&\leq \left( \sqrt{4 n P(A)\log\left( \frac{ \mathbb  S_{\mathcal A}(2n)}{\delta}\right)}  +   4 \log\left( \frac{ \mathbb  S_{\mathcal A}(2n)}{\delta}\right) \right) \wedge   \sqrt{4 n P_n(A)\log\left( \frac{ \mathbb  S_{\mathcal A}(2n)}{\delta}\right)}. 
\end{align*}  
We have, with probability at least $1-\delta$, for all $A\in \mathcal A$,
\begin{align*}
 n(P(A)- P_n (A )) & \leq \left( \sqrt{4 n P_n(A)\log\left( \frac{ (3/2)\mathbb  S_{\mathcal A}(2n)}{\delta}\right)} \right. \\
 &\left. \qquad  + 4 \log\left( \frac{ (3/2) \mathbb  S_{\mathcal A}(2n)}{\delta}\right) \right) \wedge  \sqrt{4 n P(A)\log\left( \frac{ (3/2) \mathbb  S_{\mathcal A}(2n)}{\delta}\right)}.
\end{align*}
\end{corollary}

The bounds account for low-probability sets $A$: the upper bound becomes smaller when the probability (either empirical or population) is small. They involve explicit constants and quantities available in practice. For instance, we have with probability at least $1-2\delta$,
$$ n| P_n(A)- P (A )|   \leq  \sqrt{4 n P_n(A)\log\left( \frac{ (3/2) \mathbb  S_{\mathcal A}(2n)}{\delta}\right)} + 4 \log\left( \frac{ (3/2) \mathbb  S_{\mathcal A}(2n)}{\delta}\right)  .$$
The upper bound is fully available as soon as an estimate on the shatter coefficient is given. This type of inequality, obtained for instance from the inequality of \cite{anthony1993result}, has proved useful in several statistical learning problems; see the references given in the introduction.

\section{A VC inequality for functions}\label{sec:main_results2}

We now consider the case of real valued functions. Let 
$ \mathcal F$ be a countable collection of $[0,1]$-valued functions defined on $S$. The complexity of $\mathcal F$ is controlled by the shatter coefficient of the class of subgraphs, defined as
$$  \mathrm{sg} (\mathcal F) = \{ \{ (x,y) \in S\times [0,1] \, :\,  y< f(x)   \}\,:\, f\in \mathcal F \}.$$ 
The next theorem extends Theorem \ref{theorem:th1} to class of functions (instead of sets). This is done by writing $f(x) $ as an expectation $\int \mathds{1}_{y< f(x)}\,\mathds{1}_{[0,1]}(y)\,\mathrm{d}y$ and applying twice the symmetrization used in the proof of Theorem \ref{theorem:th1}. 
The change we need is a positive shift $\gamma>0$ of the error bound making the deviation slightly weaker compared  to that of Theorem \ref{theorem:th1}. The proof only uses that $f(x)=\int_0^1\mathds 1_{\{y<f(x)\}}\,\mathrm dy$ (the same holds with $y\le f(x)$); hence Theorem~\ref{theorem:th2} holds with
$\mathrm{sg}(f)$ defined with either $y\le f(x)$ or $y<f(x)$ (with the corresponding shatter coefficient).

\begin{theorem}\label{theorem:th2}
Let $n\geq 1$ and $(Z_1,\dots,Z_n)$ be independent and identically distributed random variables with distribution $P$ on $S$. Let $\mathcal{F}$ be a countable class of measurable functions on $S$ such that $0\leq f(z)\leq 1$ for all $f\in \mathcal F$ and $z\in S$. We have, for all $t>0$ and $\gamma>0$ such that $t  \sqrt{n\gamma} \geq 16 / 9$,
$$\mathbb{P}\left(\exists f\in\mathcal{F}\,:\, P_n(f)-P(f)  >\gamma+ t\sqrt{\frac{P_n(f)}{n}} \right) \leq \mathbb{S}_{\mathrm{sg}(\mathcal{F})}(2n)\exp\left(-\frac{t^2} 4\right),$$
and
$$\mathbb{P}\left(\exists f\in\mathcal{F}\,:\, P(f)-P_n(f) >\gamma + t \sqrt{\frac{P(f)}{n}} \right) \leq \mathbb{S}_{\mathrm{sg}(\mathcal{F})}(2n)\exp\left(-\frac{t^2} 4\right).$$
\end{theorem}

The above inequality can be seen as a uniform Bernstein type inequality. For instance, when choosing $\gamma = t^2/n$ and $t= 2\sqrt {u}$, and assuming that $u \geq 4/9$, we obtain
$$\mathbb P \left(  \exists f\in \mathcal F \ \,  \  P_n(f) -P (f)   > \sqrt{\frac{4u P_n(f)}{n} }   +\frac{4u}{n}   \right) \leq \mathbb  S_{\mathrm{sg} (\mathcal F)}(2n)\exp(- u).$$
A similar statement for a single function $f$, based on the standard Bernstein inequality, is given in \cite{hsu_12}. The above statement shows that, leaving aside the constants, the price to pay for uniformity is the multiplicative factor $\mathbb  S_{\textrm{sg} (\mathcal F)}(2n)$ in the probability bound. Finally, another related inequality is given in \cite{MaurerPontil2009}, Theorem 6, where a similar data dependent deviation bound is established for all $f\in \mathcal F$, as soon as the complexity of $\mathcal F $ is suitably controlled by $L_\infty$-covering numbers. We point-out that such a restriction is not well adapted to the kernel density estimation context studied in the next section because the bandwidth parameter impact the $L_\infty$-covering numbers.  As before, we now state a corollary providing upper bounds involving either $P_n$ or $P$, depending on the situation.

\begin{corollary}\label{theorem:cor2}
In the setting of Theorem \ref{theorem:th2}, let $\delta \in (0, \exp(-1/2)]$ and set $\ell_\delta (n,\mathcal F)= \log\left( { \mathbb  S_{\textrm{sg} (\mathcal F)}(2n)}/{\delta}\right) $.  We have, with probability at least $1-\delta$, for all $f\in \mathcal F$
\begin{align*}
&n(P_n(f)- P (f )) \leq \\
&\left( \sqrt{4 n P(f)  \ell_\delta (n,\mathcal F)}  +   12 \ell_\delta (n,\mathcal F) \right) \wedge    \left( \sqrt{4 n P_n(f) \ell_\delta (n,\mathcal F) }  +   4  \ell_\delta (n,\mathcal F) \right) . 
\end{align*}
We have, with probability at least $1-\delta$, for all $f\in \mathcal F$
\begin{align*}
&n(P(f)- P_n (f )) \leq \\
&\left( \sqrt{4 n P_n(f)  \ell_\delta (n,\mathcal F)}  +   12 \ell_\delta (n,\mathcal F) \right) \wedge    \left( \sqrt{4 n P(f) \ell_\delta (n,\mathcal F) }  +   4  \ell_\delta (n,\mathcal F) \right) . 
\end{align*}

\end{corollary}

The assumption that $f$ is valued in $[0,1]$ can easily be relaxed to $f$ taking values in $[0,U]$ by changing slightly the obtained bound. The bound can also be extended to functions that may not be nonnegative provided the lower bound $f\geq -b$ holds. Applying to the translation $f+b$, we obtain
an upper-bound with a scaling factor $ P_n(f) + b$  instead of $P_n(f)$. We acknowledge that this might be detrimental when $P_n(f)$ is small compared to $b$. Relative deviation bounds for classes of functions, possibly unbounded, are investigated in \cite{cortes2019relative}.

\section{Application to kernel density estimation}\label{sec:applications}

The previous results are particularly powerful when the function class has ``nice'' subgraphs, in particular, subgraphs with small VC dimension. The VC dimension of $\mathcal A$ is defined as the largest integer $v$ such that $ \mathbb S_{\mathcal A}(v) = 2^v $ (i.e., the largest set of $v$ points that can be shattered). For such a class, the shatter coefficient grows polynomially: by the Sauer--Shelah lemma, $\mathbb S_{\mathcal A}(n) \leq (n+1)^v$ for all $n\geq 1$; see, e.g., \cite[Theorem~13.3]{devroye2013probabilistic}. This fact will be used repeatedly in this section.

Let $X_1,\ldots, X_n$ be a collection of random variables on $\mathbb R $ with common distribution $P$. We suppose further that $P$ admits a density $f_X$ with respect to Lebesgue measure. The kernel density estimator of $f_X$ is given by, for any $x\in \mathbb R$,
$$\hat f_n (x) = n^{-1} \sum_{i=1}^n K_h(x-X_i),$$
with $K_h (u) = h^{-1} K( u / h) $ for any $u\in \mathbb R$, where $h>0$ is called the bandwidth. The function $K:\mathbb R \to \mathbb R_{\geq 0}  $ is called the kernel and is often chosen as a density function (Gaussian, uniform, Epanechnikov, ...). Let $\ast $ be the standard convolution product: $(g_1\ast g_2) (x) = \int g_1(x-y) g_2(y) dy$. Define, for all $x\in \mathbb R$,
$$f_h(x) = (f_X\ast K_h) (x) .$$
While the full treatment of $\hat f_n$ would require the traditional bias-variance decomposition, we focus on bounding the variance part $\hat f_n (x)- f_h(x)$. We are interested in obtaining deviation bounds that are valid for all $n\geq 1$ and any choice of the bandwidth $h>0$. 

In what follows, we introduce two nondecreasing functions $F$ and $G$ taking values in
$[0,1]$. We control the complexity of the class of kernel functions by
assuming that, for any $u\in\mathbb R$,
\begin{align}\label{eq:kernel}
K(u) = F(u)\,\mathds 1_A(u) + (1-G(u))\,\mathds 1_{A^c}(u),
\end{align}
for a given set $A=\cup_{j=1}^q(a_j,b_j]$ with
$-\infty=a_1<b_1<a_2<b_2<\cdots<a_q<b_q<+\infty$ and for $A^c$ denoting the complementary of $A$. In this way, $F$ represents the increasing part of $K$ while
$1-G$ represents the decreasing part. One standard example is the Gaussian
kernel, obtained with $A=(-\infty,0]$, $F(u)=\exp(-u^2/2)$ and
$1-G(u)=\exp(-u^2/2)$. In this case, note that the values of $F$ (resp.\ $G$) on
$(0,+\infty)$ (resp.\ $(-\infty,0]$) do not affect $K$ and may be chosen so
that monotonicity holds on the whole real line, e.g., $F(u)=1$ for $u>0$ and
$G(u)=0$ for $u\le 0$. The next lemma shows that the shatter coefficient of
the associated class grows polynomially in $n$, as it does for a class with
finite VC dimension.

\begin{lemma}\label{lemma:comp}
Suppose that \eqref{eq:kernel} is fulfilled with $F$ and $G$ nondecreasing on
$\mathbb R$. Then, for any $h>0$, the set
$\Psi=\{z\mapsto K(h^{-1}(x-z))\,:\,x\in\mathbb R\}$ is such that, for all
$n\ge 1$,
\[
\mathbb S_{\operatorname{sg}(\Psi)}(n)\le (n+1)^{4q+2}.
\]
\end{lemma}

It is possible to extend the previous lemma to the cases where (i) neither
$F$ nor $G$ is monotone on the full real line, and
(ii) the class is indexed by the bandwidth $h>0$ as well (not only by $x$ as in Lemma \ref{lemma:comp}).
For the first extension (i), we shall only assume that $F$ is nondecreasing on
each interval of $A$ while $1-G$ is nonincreasing on each interval of $A^c$.
This comes at the cost of an increase of the shatter coefficient upper bound.

\begin{lemma}\label{lemma:comp1}
Suppose that \eqref{eq:kernel} is fulfilled with $F$ nondecreasing on each
interval of $A$ and $G$ nondecreasing on each interval of $A^c$. Then, for
any $h>0$, the set $\Psi=\{z\mapsto K(h^{-1}(x-z))\,:\,x\in\mathbb R\}$ is
such that, for all $n\ge1$,
\[
\mathbb S_{\operatorname{sg}(\Psi)}(n)\le (n+1)^{6q}.
\]
\end{lemma}

The upper bounds $(n+1)^{4q+2}$ (Lemma~\ref{lemma:comp}) and $(n+1)^{6q}$
(Lemma~\ref{lemma:comp1}) coincide only when $q=1$, in which case both equal
$(n+1)^6$ and the two settings reduce to the same one. For the second extension (ii), we need in addition that $F$ is
\emph{left-continuous} on each interval of $A$ while $G$ is
\emph{right-continuous} on each interval of $A^c$. These semicontinuity
conventions \citep[Section 7]{rockafellar1997convex} are the ones under which the sublevel sets $\{v\,:\,F(v)\le y\}$
and the superlevel sets $\{v\,:\,G(v)\ge u\}$ are closed (within each
interval); this is precisely what is used in the proof of
Lemma~\ref{lemma:comp2} below. Note that the convention on $G$ is the usual
one for cumulative distribution functions, while the one on $F$ is the
left-continuous variant; jumps are allowed for both.

\begin{lemma}\label{lemma:comp2}
Suppose that \eqref{eq:kernel} is fulfilled with $F$ nondecreasing and
left-continuous on each interval of $A$ and $G$ nondecreasing and
right-continuous on each interval of $A^c$. The set
$\Psi=\bigl\{z\mapsto K\bigl(h^{-1}(x-z)\bigr)\,:\,x\in\mathbb R,\ h>0\bigr\}$ 
is such that, for all $n\ge1$,
\[
\mathbb S_{\operatorname{sg}(\Psi)}(n)\le (n+1)^{10q}.
\]
\end{lemma}

Combining Lemmas \ref{lemma:comp1} and \ref{lemma:comp2} with Corollary \ref{theorem:cor2}, we obtain the following statement.

\begin{theorem}\label{th:density}
Let $n\geq 1$, $\delta \in  (0, \exp(-1/2)]$. Let $X_1,\ldots, X_n$ be a collection of independent and identically distributed random variables on $\mathbb R$ with density $f_X$. Suppose that \eqref{eq:kernel} is fulfilled with $F$ nondecreasing on each
interval of $A$ and $G$ nondecreasing on each interval of $A^c$. For any $h>0$, we have with probability at least $1-2\delta$, for all (countable) $x \in \mathbb R $,
$$ | \hat f_n (x) - f_h(x)  | \leq  \sqrt{ \frac{ 24q    }{nh}   \left(  \hat f_n(x)  \wedge f_h(x) \right) \log\left( \frac{ 3n}{\delta}\right)   } +   \frac{ 72q }{nh} \log\left( \frac{ 3n}{\delta}\right) .$$
If, in addition, $F$ is left-continuous on each interval of $A$ and $G$ is right-continuous on each interval of $A^c$, we have with probability at least $1-2\delta$, for all (countable) $x \in \mathbb R $ and $h>0$,
$$ | \hat f_n (x) - f_h(x)  | \leq  \sqrt{ \frac{ 40q    }{nh}   \left(  \hat f_n(x)  \wedge f_h(x) \right) \log\left( \frac{ 3n}{\delta}\right)   } +   \frac{ 120q }{nh} \log\left( \frac{ 3n}{\delta}\right) .$$

\end{theorem}

In Theorem \ref{th:density}, the quantifiers over $x\in \mathbb R$ and $h>0$ should be, rigorously speaking, understood as ranging over a countable dense subset of $\mathbb R \times (0,+\infty)$, in accordance with the countability requirement of Corollary \ref{theorem:cor2}. This is only for ensuring  measurability of the underlying supremum. Point-wise measurability \citep[Example 2.3.4]{wellner1996} might be verified to extend the previous claim to $x\in \mathbb R$ and $h>0$. If $K$ is continuous, $(x,h)\mapsto\hat f_n(x)$ and $(x,h)\mapsto f_h(x)$ are continuous (dominated convergence), hence so is
the right-hand side of the bounds, and the bounds extend from a countable dense set to all $x\in\mathbb R$, $h>0$.

The above result illustrates the strength of Theorem \ref{theorem:th2} and Corollary \ref{theorem:cor2}. Several points should be highlighted. First, it holds for any $n\geq 1$ and the constants are explicit. Second, it is valid for all $x \in \mathbb R$ with an upper bound that might be small or large depending on the density level: the higher the density, the larger the upper bound. Third, a weaker result can be obtained using the upper bound    
$$ | \hat f_n (x) - f_h(x)  | \leq  \sqrt{ \frac{ 24q   \|f_X\|_\infty  \|  K\|_1  }{nh}  \log\left( \frac{ 3n}{\delta}\right)   } +   \frac{ 72q }{nh} \log\left( \frac{ 3n}{\delta}\right) ,$$ where $\| g\|_p $ is the $L_p$-norm of $g$. Such results, with a bound uniform in the location $x$, are given in \cite{einmahl2000,gine+g:02,gine_nivkl_wave}. Our result is stronger in this comparison, since it recovers a location-independent bound while additionally adapting to the local density level; it is thus well-suited for building confidence bands for $f_h(x)$, adaptive in $x$, valid on the full real line. Setting 
$$\rho_n (x)= \sqrt{  \hat f_n (x)  }  \sqrt{ \frac{ 24q }{nh}  \log\left( \frac{ 3n }{\delta}\right)  }  +   \frac{ 72q}{nh} \log\left( \frac{ 3n}{\delta}\right), $$
 we obtain a location-adaptive, data-driven confidence band for $f_h$:
 \[
\mathbb P \left( \forall x\in \mathbb R,\ f_h (x)\in   [ \hat f_n (x) -\rho_n(x) ,  \hat f_n (x) +\rho_n(x) ] \right) \geq 1-2\delta, \] whose width matches the local standard deviation of the estimator up to constants and logarithmic factors. Consequently, the band is adaptive to the local density level, becoming narrower in low-density regions. We also provide confidence bands valid for all choices of $h$, possibly data-dependent, since we have
$$ \mathbb P \left( \forall x \in \mathbb R,\, \forall h>0, \ f_h (x)\in   [ \hat f_n (x) - \tilde \rho_n(x) ,  \hat f_n (x) + \tilde  \rho_n(x) ] \right) \geq 1-2\delta,$$
with $$\tilde \rho_n(x)  = \sqrt{  \hat f_n (x)  }  \sqrt{ \frac{ 40q }{nh}  \log\left( \frac{ 3n }{\delta}\right)  }  +   \frac{ 120q}{nh} \log\left( \frac{ 3n}{\delta}\right). $$
After combining this concentration inequality with an appropriate control of the bias term $|f_h(x)-f_X(x)|$, one may extend the confidence bands provided before to the true density $f_X$. This is left for future research. We refer to \cite{juditsky2003nonparametric,robins2006adaptive,wasserman_2008,gine_nivkl_band} for results and discussions on the construction of confidence bands in density estimation and regression.

\section{Proofs}\label{sec:proofs}

\subsection{An auxiliary convexity lemma} 

The next lemma will be key in achieving the symmetrization step needed in the proof of both Theorem \ref{theorem:th1} and Theorem \ref{theorem:th2}.

\begin{lemma}\label{lemma_conv}
The following holds:
\begin{enumerate}[label=(\roman*)]
\item Let $a \geq 0$ and $t>0$, the function $x\mapsto \exp(t {(a - x) } / {\sqrt{ x+a }} )$ is convex on $\mathbb{R}_{\geq 0}$.
\item Let $a  >0$ and $t\sqrt{ a}  \geq 7 / 9$, the function $x\mapsto \exp(t (x-a) / \sqrt{x+a} )$ is convex on $\mathbb{R}_{\geq 0}$.
\item Let $a \geq 0$, $\gamma > 0$ and $t \sqrt{ \gamma} \geq 8\sqrt 2 / 9  $, the function $x\mapsto \exp(t {(x - a-  \gamma) } / {\sqrt{ x + a }} )$ is convex on $\mathbb{R}_{\geq 0}$.
\end{enumerate}
\end{lemma}

\begin{proof}

\noindent\textit{Proof of (i).} Let $a \geq 0$. We perform the change of variable $u = x + a$, so that the function becomes $u \mapsto \exp(t( 2a - u )/\sqrt{u})$, which we need to show is convex on $[a, +\infty)$. Define $h(u) = ( 2a - u )/\sqrt{u}$ and compute
\[
h'(u) = - \frac{ 2a + u  }{2u^{3/2}}, \qquad h''(u) = \frac{u + 6a}{4u^{5/2}}.
\]
Convexity of $u \mapsto e^{t h(u)}$ is equivalent to $h''(u) + t(h'(u))^2 \geq 0$ for all $u \geq a$, which holds since $h''(u)\geq 0$ for every $u\geq a$.

\noindent\textit{Proof of (ii).} Let $a >  0$ and $t\sqrt a \geq 7/9$. We perform the change of variable $u = x + a$, so that the function becomes $u \mapsto \exp(t(u - 2a)/\sqrt{u})$, which we need to show is convex on $[a, +\infty)$. Define $h(u) = (u - 2a)/\sqrt{u}$ and compute
\[
h'(u) = \frac{u + 2a}{2u^{3/2}}, \qquad h''(u) = -\frac{u + 6a}{4u^{5/2}}.
\]
Convexity of $u \mapsto e^{t h(u)}$ is equivalent to $h''(u) + t(h'(u))^2 \geq 0$ for all $u \geq a$, that is,
\[
t \geq \sup_{u \geq a} -\frac{h''(u)}{h'(u)^2} = \sup_{u \geq a} \frac{(u+6a)\sqrt{u}}{(u+2a)^2}.
\]
Substituting $u = ay$ with $y \geq 1$, this becomes
\[
\sup_{u \geq a} \frac{(u+6a)\sqrt{u}}{(u+2a)^2} = a^{-1/2} \sup_{y \geq 1} \psi(y), \qquad \psi(y) = \frac{(y+6)\sqrt{y}}{(y+2)^2}.
\]
A direct computation gives
\[
\psi'(y) = \frac{-y^2 - 12y + 12}{2\sqrt{y}(y+2)^3}.
\]
The positive root of $y^2 + 12y - 12 = 0$ is $y^* = -6 + 4\sqrt{3} < 1$, so $-y^2 - 12y + 12 < 0$ for all $y \geq 1$, hence $\psi'(y) < 0$ on $[1,+\infty)$ and $\psi$ is decreasing there. Therefore the supremum is attained at $y = 1$:
\[
\psi(1)   = \frac{7}{9}.
\]
Therefore convexity holds whenever $t \geq a^{-1/2} ( 7/9) $, which is satisfied.

\noindent\textit{Proof of (iii).} Let $a \geq   0$, $\gamma > 0$, and $t\sqrt{\gamma} \geq 8\sqrt 2/9$. We perform the change of variable $u = x + a$, so that the function becomes $u \mapsto \exp(t(u - b)/\sqrt{u})$, with $b = 2a + \gamma$, which we need to show is convex on $[a, +\infty)$. 
Define $h(u) = (u - b)/\sqrt{u}$. Proceeding as in part (ii), we obtain
\[
h'(u) = \frac{u + b}{2u^{3/2}}, \qquad h''(u) = -\frac{u + 3b}{4u^{5/2}}.
\]
Convexity is equivalent to 
$$ t\geq  \sup_{u\geq a} - \frac{h''(u)}{h'(u)^2} =   \sup_{u\geq a}  \frac{u + 3b }{(u + b)^2} \sqrt{ u  } = (b/2)^{-1/2}  \sup_{v\geq 2a/b} \psi(v) ,$$ 
with $\psi(v) = {(v + 6) } \sqrt{ v  } / {(v + 2)^2}  $. As in step (ii), we have that $\psi : \mathbb R_{\geq 0}\to \mathbb R $ achieves its maximum at $v^*  = -6+4\sqrt 3 $ and we have $\psi(v^*)\simeq 0.778 <8/9$. A sufficient condition for convexity is then $t \sqrt{b} = t \sqrt{2a + \gamma} \geq  8\sqrt 2 / 9$ which is implied by $ t   \sqrt \gamma \geq  8 \sqrt 2 / 9$.
\end{proof}

\subsection{Proof of Theorem \ref{theorem:th1}}

 We use throughout the convention ``$0/0= 0$''. We start proving the first inequality. Let $t>0$. Since $P_n(A)>P(A)$ implies $\sqrt{P_n(A)}>\sqrt{\left(P_n(A)+P(A)\right)/2}$, we have that
\begin{align*}
 &\mathbb P \left(   \exists A\in \mathcal A\,: \,   \sqrt n  ( {P_n(A) -P (A) } )> t{\sqrt{ P_n(A)   }} \right) \\
 &\leq  \mathbb P \left(   \sup_{A\in \mathcal A}  \sqrt n   \frac{P_n(A) -P (A) }{\sqrt{ ( P_n(A) + P (A ) ) /2  }} > t \right) \\
 & =  \mathbb P \left(   \sup_{A\in \mathcal A}  t \sqrt {\frac n 2} \frac{P_n(A) -P (A) }{\sqrt{ ( P_n(A) + P (A ) )   }} > t^2/2 \right)\\
 &\leq \mathbb E \left( \exp\left(  \sup_{A\in \mathcal A}  t\sqrt {\frac n 2}  \frac{P_n(A) -P (A) }{\sqrt{ P_n(A) + P (A )     }} - t^2/2 \right)\right)\\
 &\leq \mathbb E \left(  \sup_{A\in \mathcal A}\exp\left( t\sqrt {\frac n 2} \frac{P_n(A) -P (A) }{\sqrt{ P_n(A) + P (A )    }}  \right)\right) \exp\left( - \frac{t^2 }{ 2}\right).
\end{align*}
We only have to show that
\begin{align}\label{objective}
 \mathbb E \left(  \sup_{A\in \mathcal A}\exp\left( t\sqrt {\frac n 2} \frac{P_n(A) -P (A) }{\sqrt{ P_n(A) + P (A )    }}  \right)\right) \leq \mathbb S_{\mathcal A}(2n) \exp\left(  \frac{t^2 }{ 4}\right).
\end{align}
 Let $a \in [0,1]$.   Let $g_a(x) =  {(a -x) } / {\sqrt{ (a+ x)  }}$ be
defined on $ [0,1]$. Setting  $a=P_n(A)$, note that
 $$ \exp\left( t\sqrt {\frac n 2}  \frac{P_n(A) -P (A) }{\sqrt{ (P_n(A) + P (A ))   }}\right )= \exp\left(t\sqrt {\frac n 2} g_a(  P(A) )\right)   $$ 
 Introduce a ghost sample $(Z_i')_{i=1,\ldots, n }$,  with the same distribution as $(Z_i)_{i=1,\ldots, n }$ but independent from $(Z_i)_{i=1,\ldots, n }$, and let $P_n'(A)= n^{-1} \sum_{i=1} ^n \mathrm 1_A(Z_i')$.  
 Lemma \ref{lemma_conv}, (i), shows that the function $x\mapsto  \exp(t\sqrt {n/2} g_a(x) )$  is convex on $ [0,1] $ whenever $a \geq 0$ and $t>0$. 
  As $P(A)  = \mathbb E'( P_n'(A))$, we have, 
by Jensen's inequality, 
\begin{align*}
\exp\left( t \sqrt {\frac n 2}  \frac{P_n(A) -P (A) }{\sqrt{ (P_n(A) + P (A ))    }}\right) &\leq  \mathbb E'\exp\left(t\sqrt {\frac n 2} \frac{P_n(A) -P_n'(A) }{\sqrt{ (P_n(A) + P_n'(A ))    }}\right) 
\end{align*} 
We have shown that
\begin{align*}
 \mathbb E   \sup_{A\in \mathcal A} \exp\left( t\sqrt {\frac n 2}  \frac{P_n(A) -P (A) }{\sqrt{ (P_n(A) + P (A ))     }}  \right) 
 &\leq \mathbb E\mathbb E '  \sup_{A\in \mathcal A}\exp\left(t\sqrt {\frac n 2} \frac{P_n(A) -P_n'(A) }{\sqrt{ (P_n(A) + P_n'(A ))    }} \right)
\end{align*}  
Let $W = ((Z_1,Z_1') , \ldots ,( Z_n, Z'_n) ) $ and define for each $i = 1,\ldots , n $, 
$$ \tilde W_i = (Z_i,Z_i')\mathrm 1_{\eta_i = 1} + (Z_i',Z_i)  \mathrm 1_{\eta_i = - 1},$$
where $\mathbb P (\eta_i = 1 )= 1/2 = 1-\mathbb P (\eta_i = -1)  $. Using that $W $ and $\tilde W$ have the same distribution, 
we get
\begin{align*}
&\mathbb E\mathbb E '  \sup_{A\in \mathcal A}\exp\left(t\sqrt {\frac n 2}  \frac{P_n(A) -P_n'(A) }{\sqrt{ (P_n(A) + P_n'(A ))    }}  \right) \\
 & =  \mathbb E\mathbb E '\mathbb E_\eta \sup_{A\in \mathcal A} \exp\left(t  \frac{\sum_{i=1} ^n \eta_i ( 1_A(Z_i)  -  1_A(Z_i') )  }{\sqrt{ 2n (P_n(A) + P_n'(A ))    }} \right).
\end{align*}
For fixed realizations of $(Z_1,\ldots,Z_n,Z_1',\ldots,Z_n')$, define
\[
\mathcal V=
\left\{
\left((1_A(Z_1),\ldots,1_A(Z_n)),(1_A(Z_1'),\ldots,1_A(Z_n'))\right)
:\,
A\in\mathcal A
\right\}.
\]
Then
\begin{align*}
\mathbb E_\eta \sup_{A\in \mathcal A} \exp\left(t  \frac{\sum_{i=1} ^n \eta_i ( 1_A(Z_i)  -  1_A(Z_i') )  }{\sqrt{ 2n (P_n(A) + P_n'(A ))    }} \right)
&\leq\sum_{(v,v')\in \mathcal V} \mathbb E_\eta \exp\left(t \frac{\sum_{i=1} ^n \eta_i ( v_i  -  v_i' )  }{\sqrt{2  \sum_{i=1}^n (v_i+ v_i')   }}  \right).
\end{align*}
Since $\eta_i$ is sub-Gaussian with factor $1$, and using $  \sum_{i=1} ^n (v_i - v_i' )^2 \leq     \sum_{i=1}^n (v_i+ v_i') $, we obtain 
$$\mathbb E_\eta \exp\left(t \frac{\sum_{i=1} ^n \eta_i ( v_i  -  v_i' )  }{\sqrt{2   \sum_{i=1}^n (v_i+ v_i')     }}  \right)\leq  \exp\left(t^2  \frac{\sum_{i=1} ^n (v_i - v_i' )^2  }{ {4  ( \sum_{i=1}^n (v_i+ v_i'))      }  }  \right)   \leq \exp( t^2/4).$$
Since each element of $\mathcal V$ corresponds to one configuration of the
$2n$ points $(Z_1,\ldots,Z_n)$ and $(Z_1',\ldots,Z_n')$, we have $|\mathcal V|\leq \mathbb S_{\mathcal A}(2n)$. Putting everything together leads to
\eqref{objective}.

Now we turn our attention to the second inequality whose proof is similar to the first one except that the range of admissible $t$ is reduced at some point (and extended to all $t>0$ at the end of the proof). Start writing the event of interest $ \{\exists A\in \mathcal A \,:\,\sqrt{n} (P(A)-P_n(A) ) >t \sqrt {P(A)} \}$ as a union of two events:  $\{\exists A\in \mathcal A \,:\, P_n(A) \geq 1/n,\,\sqrt{n}( P(A)-P_n(A))  >t\sqrt {P(A)} \} $   and the same event with $P_n(A) = 0$ instead of $P_n(A) \geq 1/n$. Concerning the first event, we claim that, when $t\geq 7 \sqrt 2 /9$,
\begin{align}\label{objective1}
\mathbb P \left(  \exists A\in \mathcal A\, :\, P_n(A) \geq  1/n,\,   P(A) -P_n (A)    > t\sqrt{ \frac{ P (A ) }{n}  } \right)\leq \mathbb S_{\mathcal A}(2n) \exp\left( -\frac{t^2 }{ 4}\right).
\end{align}
The proof is similar to that of the first inequality. We have
\begin{align*}
& \mathbb P \left(  \exists A\in \mathcal A\, :\, P_n(A) \geq  1/n,\, \sqrt n ( P(A) -P_n (A) )   > t\sqrt{  P (A )   } \right)\\
 & \leq \mathbb P \left(  \sup_{A\in \mathcal A\;:\,P_n(A) \geq 1/n} \sqrt n  \frac{P(A) -P_n (A) }{\sqrt{ (P_n(A) + P (A )) /2   }}  > t \right) \\
  & = \mathbb P \left(  \sup_{A\in \mathcal A\;:\,P_n(A) \geq  1/n} t \sqrt{ \frac n 2}  \frac{P(A) -P_n (A) }{\sqrt{ (P_n(A) + P (A ))    }}  > \frac{ t ^2}{ 2} \right) \\
 & \leq  \mathbb E \left(  \sup_{A\in \mathcal A\;:\,P_n(A) \geq  1/n}\exp\left( t \sqrt{ \frac n 2} \frac{P(A) -P_n (A) }{\sqrt{ (P_n(A) + P (A ))    }}  \right)\right)\exp\left( - \frac{t^2 }{ 2}\right).
\end{align*}
We apply Lemma~\ref{lemma_conv}~(ii) with $a=P_n(A)\ge1/n$ and $t\sqrt{n/2}$ in place of $t$. Since $t\sqrt{n/2}\,\sqrt{P_n(A)}\ge t/\sqrt2\ge7/9$ when $t\ge7\sqrt2/9$, the function $x\mapsto\exp\left(t\sqrt{n/2}\,(x-a)/\sqrt{x+a}\right)$ is convex on $\mathbb R_{\ge0}$. As $P(A)=\mathbb E'P_n'(A)$, Jensen's inequality gives, for every $A$ with $P_n(A)\ge1/n$,
\[
\exp\left(t\sqrt{\frac n2}\,\frac{P(A)-P_n(A)}{\sqrt{P_n(A)+P(A)}}\right)\le\mathbb E'\exp\left(t\sqrt{\frac n2}\,\frac{P_n'(A)-P_n(A)}{\sqrt{P_n(A)+P_n'(A)}}\right).
\]
Since $\sup_{A:\,P_n(A)\ge 1/n}\mathbb E'[\cdot]\le \mathbb E'\sup_{A\in\mathcal A}[\cdot]$, we have
\begin{align*}
\mathbb E\sup_{A:\,P_n(A)\ge 1/n}\exp\left(t\sqrt{\frac n2}\,\frac{P(A)-P_n(A)}{\sqrt{P_n(A)+P(A)}}\right)
&\le \mathbb E\,\mathbb E'\sup_{A\in\mathcal A}\exp\left(t\sqrt{\frac n2}\,\frac{P_n'(A)-P_n(A)}{\sqrt{P_n(A)+P_n'(A)}}\right)\\
&\le \mathbb S_{\mathcal A}(2n)\,\exp\left( \frac{t^2}{4} \right),
\end{align*}
where the last inequality is obtained exactly as in the first part, with $\eta_i$ replaced by $-\eta_i$. We have shown \eqref{objective1}. Let us now focus on the second event in which $P_n(A) = 0$. We have
\begin{align*}
\mathbb P \left(  \exists A\in \mathcal A\, :\, P_n(A) = 0,\,   P(A) -P_n (A)    > t\sqrt{ \frac{ P (A ) }{n}  } \right) 
 &= \mathbb P \left(  \sup_{A\in \mathcal A,\,P_n(A)  = 0} n   P(A)    > t^2 \right).
\end{align*}
We claim that, for every $t>0$,
\begin{align}\label{obj2}
 \mathbb P \left(  \sup_{A\in \mathcal A,\,P_n(A)  = 0} n   P(A)    > t^2 \right)\le \mathbb S_{\mathcal A}(2n)\,2^{-t^2}.
\end{align}
Our proof relies on Jensen's inequality and symmetrization as in the other steps but it is somewhat different here due to particular form of the considered event. Let $\lambda=\log 2$. By Markov's inequality the probability is at most
$2^{-t^2}\,\mathbb E\sup_{A:\,P_n(A)=0}e^{\lambda nP(A)}$.
Since $nP(A)=\mathbb E'\sum_i\mathbf 1_A(Z_i')$, Jensen's inequality gives $e^{\lambda nP(A)}\le \mathbb E'\,2^{\sum_i\mathbf 1_A(Z_i')}$, hence
\[
\mathbb E\sup_{A:\,P_n(A)=0}e^{\lambda nP(A)}\le \mathbb E\,\mathbb E'\sup_{A:\,\sum_i\mathbf 1_A(Z_i)=0}2^{\sum_i\mathbf 1_A(Z_i')}.
\]
The function inside $\mathbb E\mathbb E'$ is a function of $W=((Z_i,Z_i'))_i$, constraint included. Exchanging $Z_i\leftrightarrow Z_i'$ when $\eta_i=-1$ does not change the law of $W$. So the right-hand side equals the same expression with $\mathbb E_\eta$ in front, computed on the swapped sample. Bound the supremum by the sum over $(v,v')\in\mathcal V$.
If $v_i=v_i'=1$ for some $i$, the swapped first sample is never empty and the term is $0$.
Otherwise let $r=\sum_i(v_i+v_i')$. The swapped first sample is empty iff $\eta_i=-1$ whenever $v_i=1$ and $\eta_i=+1$ whenever $v_i'=1$. This has probability $2^{-r}$, and on this event the swapped second sample contains exactly $r$ points of $A$. Hence $\mathbb E_\eta[\cdot]=2^{-r}2^{r}=1$ for each pattern, and the total is at most $|\mathcal V|\le\mathbb S_{\mathcal A}(2n)$.

From \eqref{objective1} and \eqref{obj2}, it follows that for all $t\geq  7\sqrt 2 /9$, 
$$ \mathbb P \left(   
\exists A\in \mathcal A \,:\,  P(A)-P_n(A)  >t \sqrt {\frac{P(A)}{n}} \right)\leq
\mathbb S_{\mathcal A}(2n) \left( \exp\left( - \frac{t^2 }{ 4}\right) + 2^{-t^2} \right).$$
Whenever $t<  7\sqrt 2 /9$, the upper-bound is always larger than $1$, making the inequality valid for all $t>0$. To conclude, write $e^{-t^2/4}+2^{-t^2}=e^{-t^2/4}\left(1+r(t)\right)$ with $r(t)=e^{-(\ln2-1/4)t^2}$, which is decreasing in $t$.
If $t^2\le4\ln(3/2)$, then $(3/2)\,\mathbb S_{\mathcal A}(2n)\,e^{-t^2/4}\ge1$ and the claimed bound is trivial.
If $t^2>4\ln(3/2)$, $r(t)<(3/2)^{-(4\ln2-1)}<1/2$, whence $1+r(t)<3/2$.

\qed

\subsection{Proof of Corollary \ref{theorem:cor1}}

The first inequality gives, with probability at least $1-\delta$, for all $A\in \mathcal A$
$$ n ({P_n(A) -P (A) } )\leq {\sqrt{ 4 nP_n (A )   \log\left( \frac{ \mathbb  S_{\mathcal A}(2n)}{\delta}\right)}    }.$$
Introducing $X= n(P_n(A)- P (A ))$, it follows that
$$ X  \leq {\sqrt{ 4  (n P(A) +X  ) \log\left( \frac{\mathbb  S_{\mathcal A}(2n)}{\delta}\right)}    }$$
or equivalently, $X^2 \leq   4 X \log(  \mathbb S_{\mathcal A}(2n)/\delta) +  4 n P(A) \log(   \mathbb S_{\mathcal A}(2n)/\delta) $. Solving for the highest root (and using $\sqrt{a+b}\le\sqrt a+\sqrt b$), it follows that $$X\leq 4  \log\left( \frac{\mathbb  S_{\mathcal A}(2n)}{\delta}\right)+  \sqrt{4 n P(A) \log\left( \frac{\mathbb  S_{\mathcal A}(2n)}{\delta}\right)}.$$
 The second inequality implies that, with probability at least $1-\delta$, for all $A\in \mathcal A$
$$  n ({P(A) -P _n (A) } )\leq {\sqrt{ 4 nP (A )  \log\left( \frac{ (3/2) \mathbb  S_{\mathcal A}(2n)}{\delta}\right) }    }.$$
To obtain the second statement, we proceed as before, exchanging the roles of $P$ and $P_n$ and replacing $\mathbb  S_{\mathcal A}(2n)$ by $ (3/2)\,\mathbb  S_{\mathcal A}(2n)$ in the logarithmic factor.

\qed

\subsection{Proof of Theorem \ref{theorem:th2}}
We start proving the first inequality. The second will be established in a similar manner. Ratios with vanishing denominator are read with $0/0=0$ and $c/0=-\infty$ for $c<0$ (the corresponding exponentials are $0$). Let $t>0$, $\gamma >0$ and $n\geq 1$ be such that $t\sqrt  { \gamma n }\geq 16 / 9$. Note that, since $P_n(f)>P(f)+\gamma$ implies $P_n(f)>(P_n(f)+P(f))/2>0$, it holds
\begin{align*}
&\mathbb P\left(\exists f\in\mathcal F:\ P_n(f)-P(f)>\gamma+t\sqrt{P_n(f)/n}\right)\\
&\le\mathbb P\left(\exists f\in\mathcal F:\ t\sqrt{n/2}\,\frac{P_n(f)-P(f)-\gamma}{\sqrt{P_n(f)+P(f)}}>\frac{t^2}2\right)\\
 &\leq \mathbb P \left(   \sup_{f\in \mathcal F} t \sqrt {\frac n 2}  \frac{P_n(f) -P (f) - \gamma  }{\sqrt{ P_n(f) + P (f )}    }- t^2/2 >0 \right)\\
 &\leq \mathbb E \left( \exp\left(   \sup_{f\in \mathcal F} t \sqrt {\frac n 2}  \frac{P_n(f) -P (f) - \gamma  }{\sqrt{ P_n(f) + P (f )}    }- t^2/2 \right)\right)\\
 & \leq  \mathbb E \left(  \sup_{f\in \mathcal F}\exp\left( t\sqrt {\frac n 2} \frac{P_n(f) -P (f)- \gamma }{\sqrt{ P_n(f) + P (f )    }} \right)\right) \exp(-t^2/2).
\end{align*}
  We only need to show that
\begin{align}\label{objective_1}
   \mathbb E \left(  \sup_{f\in \mathcal F}\exp\left( t\sqrt {\frac n 2} \frac{P_n(f) -P (f)- \gamma }{\sqrt{ P_n(f) + P (f )    }} \right)\right)\leq \mathbb  S_{\textrm{sg} (\mathcal F)}(2n) \exp(t^2/4).
\end{align} 
 Let $Y$ be a random variable with uniform distribution on $(0,1)$. Because $f$ is valued in $[0,1]$, we have  $ f(x) =  E_Y 1_{A}(x,Y)$ with $A = \{ (x,y)\in S\times [0,1] \,: \, y<  f(x)  \}$. We now introduce several random variables. Let $(Y_i)_{i=1,\ldots, n } $ be a collection of independent and identically distributed random variables with uniform distribution on $(0,1)$. We suppose that $(Y_i)_{i=1,\ldots, n } $ is independent from $(Z_i)_{i=1,\ldots, n }$. Let $ (Z_i' ,Y_i')_{i=1,\ldots, n }  $ be independent of  $(Z_i, Y_i)_{i=1,\ldots, n }  $ with the same distribution as $(Z_i , Y_i)_{i=1,\ldots, n }  $. The expectation $\mathbb E'$ is over the collection $(Z_i')_{i=1,\ldots, n }, (Y_i')_{i=1,\ldots, n } $ while keeping the other $(Z_i,Y_i)_{i=1,\ldots, n } $ as fixed. The expectation $\mathbb E _{Y_{(1:n)}}$ is over the collection $ (Y_i)_{i=1,\ldots, n } $ while keeping  $(Z_i, Z_i', Y_i')_{i=1,\ldots, n } $ as fixed.  Equipped with this notation, we have
 $$ P(f) = \mathbb E ' n^{-1}\sum_{i=1}^n  1_{A}(Z'_i ,Y'_i) =  \mathbb E' P_n'(A)$$
 as well as
 $$P_n(f) =  n^{-1}\sum_{i=1}^n  f(Z_i) = \mathbb E _{Y_{(1:n)}} n^{-1}\sum_{i=1}^n  1_{A}(Z_i ,Y_i) =  \mathbb E _{Y_{(1:n)}} P_n(A).$$
It follows that
$$ \frac{P_n (f) - P(f)- \gamma }{\sqrt{ P_n(f) + P (f )  } } = \frac{   \mathbb E _{Y_{(1:n)}} P_n(A) -  \mathbb E' P_n'(A) - \gamma   }{\sqrt{\mathbb E _{Y_{(1:n)}} P_n(A) +  \mathbb E' P_n'(A)   } } .$$
Use  Lemma \ref{lemma_conv}, (iii) - convexity of $x\mapsto  \exp(t ( {x-a - \gamma })/{\sqrt{ (a+x)   }})$, valid because $t\sqrt { n / 2} \sqrt{\gamma} \geq  8\sqrt 2 / 9$ -  to obtain
 \begin{align*}
\exp\left( t\sqrt {\frac n 2}\frac{P_n (f) - P(f) - \gamma}{\sqrt{ P_n(f) + P (f )} }  \right) &=   \exp\left(t\sqrt {\frac n 2}  \frac{   \mathbb E _{Y_{(1:n)}} P_n(A) -  \mathbb E' P_n'(A) -\gamma }{\sqrt{\mathbb E _{Y_{(1:n)}}P_n(A) +  \mathbb E' P_n'(A)  } } \right) \\
& \leq  \mathbb E _{Y_{(1:n)}}  \exp\left(t\sqrt {\frac n 2}  \frac{   P_n(A) -  \mathbb E' P_n'(A) - \gamma   }{\sqrt{ P_n(A) + \mathbb E'  P_n'(A)  } } \right)  .
\end{align*} 
Then, removing the $\gamma>0 $, which diminishes the quantity of interest, we get
 \begin{align*}
&   \exp\left( t\sqrt {\frac n 2}\frac{P_n (f) - P(f) - \gamma}{\sqrt{ P_n(f) + P (f )} }  \right) \leq
  \mathbb E _{Y_{(1:n)}}  \exp\left(t\sqrt {\frac n 2}  \frac{   P_n(A) - \mathbb E' P_n'(A)    }{\sqrt{ P_n(A) +  \mathbb E'  P_n'(A)  } } \right)  .
\end{align*}  
 Use Lemma \ref{lemma_conv}, (i) - convexity of $x\mapsto  \exp(t ( {a -x })/{\sqrt{ (a+x)   }})$ - to obtain
 \begin{align*}
\exp\left( t\sqrt {\frac n 2}\frac{P_n (f) - P(f) - \gamma}{\sqrt{ P_n(f) + P (f )} }  \right) 
&\leq   \mathbb E _{Y_{(1:n)}}  \mathbb E'   \exp\left(t\sqrt {\frac n 2}  \frac{   P_n(A) -  P_n'(A)  }{\sqrt{P_n(A) +   P_n'(A)  } } \right) . 
\end{align*} 
Consequently, noting from now on $A_f$ in place of $A$, we have
\begin{align*}
& \sup_{f\in \mathcal F} \exp\left( t\sqrt {\frac n 2}\frac{P_n (f) - P(f) - \gamma}{\sqrt{ P_n(f) + P (f )} }  \right)  \\&\leq    \mathbb E _{Y_{(1:n)}} \mathbb E'   \sup_{f\in \mathcal F}  \exp\left(t\sqrt {\frac n 2}  \frac{   P_n(A_f) -  P_n'(A_f)   }{\sqrt{P_n(A_f) +   P_n'(A_f) } } \right) .
\end{align*}
Let $W = ( ((Z_1,Y_1),(Z_1',Y_1')) , \ldots ,( (Z_n,Y_n), (Z'_n,Y_n') ) ) $ and define for each $i = 1,\ldots , n $, 
$$ \tilde W_i = ((Z_i,Y_i),(Z_i',Y_i'))\mathrm 1_{\eta_i = 1} + ((Z_i',Y_i'),(Z_i,Y_i))  \mathrm 1_{\eta_i = - 1},$$
where $\mathbb P (\eta_i = 1 )= 1/2 = 1-\mathbb P (\eta_i = -1)  $. Using that $W $ and $\tilde W = (\tilde W_1,\ldots, \tilde W_n)$ have the same distribution, 
we get
\begin{align*}
&\mathbb E\mathbb E _{Y_{(1:n)}} \mathbb E '  \sup_{f\in \mathcal F}\exp\left(t\sqrt {\frac n 2}  \frac{P_n(A_f) -P_n'(A_f)  }{\sqrt{ (P_n(A_f) + P_n'(A_f ))    }}  \right) \\
 & =  \mathbb E \mathbb E _{Y_{(1:n)}} \mathbb E '   \mathbb E_\eta \sup_{f\in \mathcal F} \exp\left(t  \frac{ \sum_{i=1} ^n \eta_i ( 1_{A_f}(Z_i,Y_i)  -  1_{A_f}(Z_i',Y_i') )  }{\sqrt{ 2n (P_n(A_f) + P_n'(A_f ))    }} \right)
\end{align*}
For fixed realizations of $((Z_1,Y_1),\ldots,(Z_n,Y_n),(Z_1',Y_1'),\ldots,(Z_n',Y_n'))$, define
\[
\mathcal V=
\left\{
\left((1_{A_f}(Z_1,Y_1),\ldots,1_{A_f}(Z_n,Y_n)),(1_{A_f}(Z_1',Y_1'),\ldots,1_{A_f}(Z_n',Y_n'))\right)
:\,
f\in\mathcal F
\right\}.
\]
Then
\begin{align*}
&\mathbb E_\eta \sup_{f\in \mathcal F} \exp\left(t  \frac{\sum_{i=1} ^n \eta_i ( 1_{A_f}(Z_i,Y_i)  -  1_{A_f}(Z_i',Y_i') )  }{\sqrt{ 2n (P_n({A_f}) + P_n'({A_f} ))    }} \right)\\
&\leq\sum_{(v,v')\in \mathcal V} \mathbb E_\eta \exp\left(t \frac{\sum_{i=1} ^n \eta_i ( v_i  -  v_i' )  }{\sqrt{2  \sum_{i=1}^n (v_i+ v_i')   }}  \right).
\end{align*}
Since $\eta_i$ is sub-Gaussian with factor $1$, and using $  \sum_{i=1} ^n (v_i - v_i' )^2 \leq     \sum_{i=1}^n (v_i+ v_i') $, we obtain 
$$\mathbb E_\eta \exp\left(t \frac{\sum_{i=1} ^n \eta_i ( v_i  -  v_i' )  }{\sqrt{2   \sum_{i=1}^n (v_i+ v_i')     }}  \right)\leq  \exp\left(t^2  \frac{\sum_{i=1} ^n (v_i - v_i' )^2  }{ {4  ( \sum_{i=1}^n (v_i+ v_i'))      }  }  \right)   \leq \exp( t^2/4).$$
Since each element of $\mathcal V$ corresponds to one configuration of the
$2n$ points $(Z_1,Y_1),\ldots,(Z_n,Y_n)$ and $(Z_1',Y_1'),\ldots,(Z_n',Y_n')$, we have $|\mathcal V|\leq \mathbb S_{sg(\mathcal F) }(2n)$. Putting everything together leads to \eqref{objective_1}. 

To prove the second inequality, we follow the same path except that the main step consists in obtaining
\begin{align}\label{objective_2}
   \mathbb E \left(  \sup_{f\in \mathcal F}\exp\left( t\sqrt {\frac n 2} \frac{ P (f) - P_n(f) - \gamma }{\sqrt{ P_n(f) + P (f )    }} \right)\right)\leq \mathbb  S_{\textrm{sg} (\mathcal F)}(2n) \exp(t^2/4).
\end{align} 
For this we write
$$ \frac{ P(f)- P_n (f) -  \gamma }{\sqrt{ P_n(f) + P (f )  } } = \frac{     \mathbb E' P_n'(A) - \mathbb E _{Y_{(1:n)}} P_n(A) -  \gamma   }{\sqrt{\mathbb E _{Y_{(1:n)}} P_n(A) +  \mathbb E' P_n'(A)   } } ,$$
and then the convexity results of Lemma \ref{lemma_conv}, (iii) and next (i) after dropping out the $\gamma$, are used in a similar fashion as before. The randomization step can also be done and we obtain \eqref{objective_2} by relying on the sub-Gaussian property of the Rademacher variables.

\subsection{Proof of Corollary \ref{theorem:cor2}}

Set $\ell=\log\left(\mathbb S_{\mathrm{sg}(\mathcal F)}(2n)/\delta\right)$, $t=2\sqrt\ell$ and $\gamma=t^2/n$,
so that $\mathbb S_{\mathrm{sg}(\mathcal F)}(2n)e^{-t^2/4}=\delta$. Since $\delta\le e^{-1/2}$, $\ell\ge1/2$ and
$t\sqrt{n\gamma}=t^2=4\ell\ge2\ge16/9$. Hence, with probability at least $1-\delta$, for all $f\in\mathcal F$,
$$ n(P_n(f) -P (f) ) \leq \sqrt{ 4 n  P_n(f) \log \left( \frac{\mathbb  S_{\textrm{sg} (\mathcal F)}(2n) }{ \delta}  \right)} +  4  \log\left (\frac{\mathbb  S_{\textrm{sg} (\mathcal F)}(2n) }{ \delta} \right).$$
With $X = n(P_n(f) -P (f) ) - 4 \ell  $, it gives
$$ X\leq \sqrt{ 4 (X+ 4\ell +n  P (f))  \ell } .$$
and then 
$$ X^2 \leq   4 X \ell    +   4(4 \ell +n  P (f))  \ell $$
solving we obtain 
$$ X\leq  4 \ell  + \sqrt{4(4 \ell +n  P (f))  \ell }$$
As a consequence, using that $\sqrt {a+b}\leq \sqrt a + \sqrt b $, 
$$ n(P_n(f) -P (f) )  \leq 12\ell +\sqrt{ 4 n  P (f) \ell }.$$
Putting this together with the first bound gives the first stated results. The proof of the second result is the same replacing $P_n$ with $P$ and conversely.
\qed

\subsection{Proof of Lemma \ref{lemma:comp}}

\noindent\emph{Preliminaries.} Let $\psi(u)=K(u/h)$. We have $\psi=F\mathds 1_A+(1-G)\mathds 1_{A^c}$, where $F$, $G$
and the set $A$ are re-scaled versions of the corresponding objects in
\eqref{eq:kernel}. In particular, $F$ and $G$ are nondecreasing $[0,1]$-valued functions and
$A=\bigcup_{j=1}^q(a_j,b_j]$ with the convention $a_1=-\infty$, so that
$A^c=\bigcup_{j=1}^{q}(b_j,a_{j+1}]$ with the convention $a_{q+1}=+\infty$ and the
reading $(b_q,+\infty]\cap\mathbb R=(b_q,+\infty)$; that is, $A^c$ is also a union of
$q$ intervals of the same type.

\noindent\emph{The decomposition.} Observe that $\operatorname{sg}(\Psi)$ consists of the sets $\{(z,y)\in\mathbb R\times[0,1]:\ y<\psi(x-z)\}$, $x\in\mathbb R$.
Their complements in $\mathbb R\times[0,1]$ are the sets $E_x=\{(z,y):\ \psi(x-z)\le y\}$. A class of sets and the
class of its complements have the same shatter coefficients, so $\mathbb S_{\operatorname{sg}(\Psi)}(n)=\mathbb S_{\mathcal E}(n)$
with $\mathcal E=\{E_x:\ x\in\mathbb R\}$, and it suffices to bound the latter. Moreover, $\psi(x-z)\le y$ if and only if either $F(x-z)\le y$ and $x-z\in A$, or
$G(x-z)\ge 1-y$ and $x-z\in A^c$. Let
\begin{align*}
&\mathcal I
=
\left\{
\{(z,y)\in\mathbb R\times[0,1]\,:\,x-z\in A,\ F(x-z)\le y\}\,:\,x\in\mathbb R
\right\},\\
&\mathcal J
=
\left\{
\{(z,y)\in\mathbb R\times[0,1]\,:\,x-z\in A^c,\ G(x-z)\ge 1-y\}\,:\,x\in\mathbb R
\right\}.
\end{align*}
Each element of $\mathcal E $ is the union of the element of $\mathcal I$
and the element of $\mathcal J$ associated with the \emph{same} value of $x$, whence
\[
\mathcal E
\subset
\{I\cup J\,:\,I\in\mathcal I,\ J\in\mathcal J\}.
\]
Note further that
$\mathcal I\subset\{I_1\cap I_2\,:\,I_1\in\mathcal I_1,\ I_2\in\mathcal I_2\}$ and
$\mathcal J\subset\{J_1\cap J_2\,:\,J_1\in\mathcal J_1,\ J_2\in\mathcal J_2\}$, with
\begin{align*}
&\mathcal I_1
=
\left\{
\{(z,y)\in\mathbb R\times[0,1]\,:\,x-z\in A\}\,:\,x\in\mathbb R
\right\},\\
&\mathcal I_2
=
\left\{
\{(z,y)\in\mathbb R\times[0,1]\,:\,F(x-z)\le y\}\,:\,x\in\mathbb R
\right\},\\
&\mathcal J_1
=
\left\{
\{(z,y)\in\mathbb R\times[0,1]\,:\,x-z\in A^c\}\,:\,x\in\mathbb R
\right\},\\
&\mathcal J_2
=
\left\{
\{(z,y)\in\mathbb R\times[0,1]\,:\,G(x-z)\ge 1-y\}\,:\,x\in\mathbb R
\right\}.
\end{align*}
Since shatter coefficients are nondecreasing with respect to inclusion of classes,
Theorem~13.5 of \cite{devroye2013probabilistic} (intersection and union of classes)
yields
\begin{equation}\label{eq:product-bound}
\mathbb S_{\mathcal E}(n)
\le
\mathbb S_{\mathcal I_1}(n)\,\mathbb S_{\mathcal I_2}(n)\,
\mathbb S_{\mathcal J_1}(n)\,\mathbb S_{\mathcal J_2}(n).
\end{equation}
We now bound each factor.

\noindent\emph{Bounds on $\mathbb S_{\mathcal I_2}(n)$ and $\mathbb S_{\mathcal J_2}(n)$.}
For $x\le x'$ we have $F(x-z)\le F(x'-z)$ for all $z$, because $F$ is nondecreasing.
Consequently,
\[
\{(z,y)\,:\,F(x'-z)\le y\}\subset\{(z,y)\,:\,F(x-z)\le y\},
\]
so that $\mathcal I_2$ is linearly ordered by inclusion. A class of sets that is
linearly ordered by inclusion has VC dimension at most $1$: if $\{p_1,p_2\}$ were
shattered, some set of the class would contain $p_1$ but not $p_2$ and another would
contain $p_2$ but not $p_1$, and these two sets would not be comparable; see also \cite[Section~2.6]{wellner1996}. By Sauer's lemma
\cite[Theorem~13.3]{devroye2013probabilistic}, it follows that
$\mathbb S_{\mathcal I_2}(n)\le n+1$. The same argument applies to $\mathcal J_2$: for
$x\le x'$, $G(x-z)\le G(x'-z)$, hence
$\{(z,y)\,:\,G(x'-z)\ge 1-y\}\supset\{(z,y)\,:\,G(x-z)\ge 1-y\}$, so that
$\mathcal J_2$ is also linearly ordered by inclusion and
$\mathbb S_{\mathcal J_2}(n)\le n+1$.

\noindent\emph{Bound on $\mathbb S_{\mathcal I_1}(n)$.}
Since $x- A =\bigcup_{j=1}^q \{ x -  ( a_j, b_j]\} $ (with $x- a_1=+\infty$), we have
\[
\mathcal I_1
=\{( x - A )\times[0,1]\,:\,x\in\mathbb R\}
\subset
\left\{\left(\bigcup_{j=1}^q A_j\right)\times[0,1]\,:\,A_j\in\mathcal R,\
\forall j=1,\ldots,q\right\},
\]
where $\mathcal R=\{[ a,b ) \cap\mathbb R\,:\,a\in[-\infty,+\infty),\ b\in(-\infty,+\infty]\}$
is the class of all right-open, left-closed intervals, including half-lines and
$\mathbb R$ itself. The class $\mathcal R$ has VC dimension $2$: no three points
$p_1<p_2<p_3$ can be shattered, since no interval contains $p_1$ and $p_3$ without
containing $p_2$, while two points are clearly shattered. Hence, by Sauer's lemma,
$\mathbb S_{\mathcal R}(n)\le(n+1)^2$, and applying the union property of
Theorem~13.5 of \cite{devroye2013probabilistic} repeatedly ($q-1$ times), we obtain
\[
\mathbb S_{\mathcal I_1}(n)\le\prod_{j=1}^q\mathbb S_{\mathcal R}(n)\le(n+1)^{2q}.
\]

\noindent\emph{Bound on $\mathbb S_{\mathcal J_1}(n)$.}
Each element of $\mathcal J_1$ is the complement in $\mathbb R\times[0,1]$ of the
corresponding element of $\mathcal I_1$. Since a class of sets and the class of its
complements shatter exactly the same finite sets, their shatter coefficients coincide,
and therefore
\[
\mathbb S_{\mathcal J_1}(n)=\mathbb S_{\mathcal I_1}(n)\le(n+1)^{2q}.
\]

\noindent \emph{Conclusion.} Plugging the four bounds into \eqref{eq:product-bound} gives
\[
\mathbb S_{\operatorname{sg}(\Psi)}(n) =\mathbb S_{\mathcal E}(n)
\le
(n+1)^{2q}\,(n+1)\,(n+1)^{2q}\,(n+1)
=(n+1)^{4q+2}.
\]
\qed

\subsection{Proof of Lemma \ref{lemma:comp1}}

\noindent\emph{Preliminaries.} Let $\psi(u)=K(u/h)$. We have
$\psi=\tilde F\mathds 1_{\tilde A}+(1-\tilde G)\mathds 1_{\tilde A^c}$, where
$\tilde F(u)=F(u/h)$, $\tilde G(u)=G(u/h)$ and $\tilde A=hA$; since $h>0$, the
set $\tilde A$ is again a union of $q$ intervals of the form
$(\tilde a_j,\tilde b_j]$ with $\tilde a_1=-\infty$, and $\tilde F$ (resp.\
$\tilde G$) is nondecreasing on each interval of $\tilde A$ (resp.\ of
$\tilde A^c$). To lighten notation, we write $F$, $G$, $A$, $a_j$, $b_j$ for
these re-scaled objects in the rest of the proof.

\noindent\emph{The decomposition.} As in the proof of
Lemma~\ref{lemma:comp}, we use $\mathbb S_{\operatorname{sg}(\Psi)}(n) =\mathbb S_{\mathcal E}(n)$ and that $\mathcal E$ consists of the sets of
points $(z,y)\in\mathbb R\times[0,1]$ satisfying $\psi(x-z)\le y$, for some
$x\in\mathbb R$, and $\psi(x-z)\le y$ if and only if either $F(x-z)\le y$ and
$x-z\in A$, or $G(x-z)\ge 1-y$ and $x-z\in A^c$. Since $F$ and $G$ are no
longer monotone on the whole real line, the global factorization used in the
proof of Lemma~\ref{lemma:comp} is not available, and we instead decompose
interval by interval. For each $j\in\{1,\ldots,q\}$, define the functions
$F_j,G_j:\mathbb R\to[0,1]$ by
\[
F_j(v)=
\begin{cases}
\lim_{u\downarrow a_j}F(u) & v\le a_j,\\
F(v) & v\in(a_j,b_j],\\
F(b_j) & v>b_j,
\end{cases}
\qquad
G_j(v)=
\begin{cases}
\lim_{u\downarrow b_j}G(u) & v\le b_j,\\
G(v) & v\in(b_j,a_{j+1}],\\
G(a_{j+1}) & v>a_{j+1},
\end{cases}
\]
where the third case in the definition of $G_q$ is void since
$a_{q+1}=+\infty$ (the limits exist by monotonicity and boundedness). Each
$F_j$ (resp.\ $G_j$) is nondecreasing on all of $\mathbb R$, as follows from
the monotonicity of $F$ on $(a_j,b_j]$ (resp.\ of $G$ on $(b_j,a_{j+1}]$)
alone. Since $F_j=F$ on $(a_j,b_j]$ and $G_j=G$ on $(b_j,a_{j+1}]$, we
obtain, for every $x\in\mathbb R$,
\begin{align*}
&\{(z,y)\,:\,\psi(x-z)\le y\}\\
&= \bigcup_{j=1}^q\left(
\{(z,y)\,:\,x-z\in(a_j,b_j]\}\cap\{(z,y)\,:\,F_j(x-z)\le y\}\right)\\
&
\qquad \cup \bigcup_{j=1}^{q}\left(
\{(z,y)\,:\,x-z\in(b_j,a_{j+1}]\}\cap\{(z,y)\,:\,G_j(x-z)\ge 1-y\}\right).
\end{align*}
Consequently,
\[
\mathcal E
\subset
\left\{
\bigcup_{j=1}^q (A_j\cap B_j)\ \cup\ \bigcup_{j=1}^{q} (C_j\cap D_j)
\,:\,
A_j,C_j\in\mathcal R',\ B_j\in\mathcal I^{(j)},\ D_j\in\mathcal J^{(j)}
\right\},
\]
where $\mathcal R'=\{I\times[0,1]\,:\,I\in\mathcal R\}$, with $\mathcal R$
the class of all right-open, left-closed intervals introduced in the proof of
Lemma~\ref{lemma:comp}, and
\begin{align*}
&\mathcal I^{(j)}
=\left\{\{(z,y)\in\mathbb R\times[0,1]\,:\,F_j(x-z)\le y\}\,:\,x\in\mathbb R\right\},\\
&\mathcal J^{(j)}
=\left\{\{(z,y)\in\mathbb R\times[0,1]\,:\,G_j(x-z)\ge 1-y\}\,:\,x\in\mathbb R\right\}.
\end{align*}
Since shatter coefficients are nondecreasing with respect to inclusion of
classes, the intersection and union properties of
Theorem~13.5 of \cite{devroye2013probabilistic}, applied repeatedly, yield
\begin{equation}\label{eq:product-bound-piecewise}
\mathbb S_{\mathcal E}(n)
\le
\mathbb S_{\mathcal R'}(n)^{2q}\,
\prod_{j=1}^q\mathbb S_{\mathcal I^{(j)}}(n)\,
\prod_{j=1}^{q}\mathbb S_{\mathcal J^{(j)}}(n).
\end{equation}

\noindent\emph{Bounds on each factor.} It was established in the proof of
Lemma~\ref{lemma:comp} that $\mathcal R$ has VC dimension $2$; the same holds
for $\mathcal R'$, whose elements are cylinders over those of $\mathcal R$,
so that, by Sauer's lemma, $\mathbb S_{\mathcal R'}(n)\le(n+1)^2$. Moreover,
for each fixed $j$, the function $F_j$ (resp.\ $G_j$) is nondecreasing on all
of $\mathbb R$, so the classes $\mathcal I^{(j)}$ and $\mathcal J^{(j)}$ are
of exactly the type of the classes $\mathcal I_2$ and $\mathcal J_2$ handled
in the proof of Lemma~\ref{lemma:comp}: the same argument shows that they are
linearly ordered by inclusion, hence of VC dimension at most $1$, and
$\mathbb S_{\mathcal I^{(j)}}(n)\le n+1$,
$\mathbb S_{\mathcal J^{(j)}}(n)\le n+1$.

\noindent\emph{Conclusion.} Plugging these bounds into
\eqref{eq:product-bound-piecewise} gives
\[
\mathbb S_{\operatorname{sg}(\Psi)}(n) =\mathbb S_{\mathcal E}(n)
\le
(n+1)^{4q}\,(n+1)^{q}\,(n+1)^{q}
=(n+1)^{6q}.
\]
\qed

\subsection{Proof of Lemma \ref{lemma:comp2}}

\noindent\emph{Preliminaries.} Since the bandwidth is now a parameter of the
class, no re-scaling is performed and we work with the original objects $F$,
$G$, $A$ of \eqref{eq:kernel}. Let $F_j$ and $G_j$, $j=1,\ldots,q$, be the
clamped extensions defined in the proof of Lemma~\ref{lemma:comp1}. In
addition to being nondecreasing on all of $\mathbb R$, each $F_j$ is
left-continuous on all of $\mathbb R$: on $(a_j,b_j]$ this is the assumption
on $F$, and elsewhere $F_j$ is locally constant to the left of every point.
Symmetrically, each $G_j$ is right-continuous on all of $\mathbb R$.

\noindent\emph{The decomposition.} As in the proof of Lemma~\ref{lemma:comp}, $\mathbb S_{\operatorname{sg}(\Psi)}(n)=\mathbb S_{\mathcal E}(n)$,
where $\mathcal E$ contains all complements, i.e., $$\mathcal E=\left\{\left\{(z,y)\in\mathbb R\times[0,1]:\ K\left((x-z)/h\right)\le y\right\}:\ x\in\mathbb R,\ h>0\right\}.$$
Writing $v=(x-z)/h$ and repeating the argument of the proof of Lemma~\ref{lemma:comp1} for each pair $(x,h)$, we obtain 
\[
\mathcal E
\subset
\left\{
\bigcup_{j=1}^q (A_j\cap B_j)\ \cup\ \bigcup_{j=1}^{q} (C_j\cap D_j)
\,:\,
A_j,C_j\in\mathcal R',\ B_j\in\mathcal I^{(j)},\ D_j\in\mathcal J^{(j)}
\right\},
\]
where $\mathcal R'$ is as before (indeed, since $h>0$, we have
$(x-z)/h\in(a,b]$ if and only if $z\in[x- hb,\,x- ha)$, which is again an
element of $\mathcal R$), and where the per-interval classes are now indexed
by both parameters:
\begin{align*}
&\mathcal I^{(j)}
=\left\{\left\{(z,y)\in\mathbb R\times[0,1]\,:\,
F_j\left(\frac{x-z}{h}\right)\le y\right\}\,:\,x\in\mathbb R,\ h>0\right\},\\
&\mathcal J^{(j)}
=\left\{\left\{(z,y)\in\mathbb R\times[0,1]\,:\,
G_j\left(\frac{x-z}{h}\right)\ge 1-y\right\}\,:\,x\in\mathbb R,\ h>0\right\}.
\end{align*}
As in the proof of Lemma~\ref{lemma:comp1}, the intersection and union
properties of Theorem~13.5 of \cite{devroye2013probabilistic} yield
\begin{equation}\label{eq:product-bound-piecewise-h}
\mathbb S_{\mathcal E}(n)
\le
\mathbb S_{\mathcal R'}(n)^{2q}\,
\prod_{j=1}^q\mathbb S_{\mathcal I^{(j)}}(n)\,
\prod_{j=1}^{q}\mathbb S_{\mathcal J^{(j)}}(n),
\end{equation}
and $\mathbb S_{\mathcal R'}(n)\le(n+1)^2$ as established there. It remains
to bound the shatter coefficients of $\mathcal I^{(j)}$ and
$\mathcal J^{(j)}$; the chain argument of the previous proofs is no longer
available, because with two free parameters these classes are not linearly
ordered by inclusion, and we argue instead through the generalized inverses
of $F_j$ and $G_j$.

\noindent\emph{Bound on $\mathbb S_{\mathcal I^{(j)}}(n)$.}
Fix $j\in\{1,\ldots,q\}$ and define
$L_j(y)=\sup\{v\in\mathbb R\,:\,F_j(v)\le y\}\in[-\infty,+\infty]$, with the
convention $\sup\emptyset=-\infty$. Since $F_j$ is nondecreasing, the set
$\{v\,:\,F_j(v)\le y\}$ is a left half-line; since $F_j$ is left-continuous,
this set is closed whenever it is nonempty and proper: indeed, if
$L_j(y)\in\mathbb R$, then
$F_j(L_j(y))=\lim_{v\uparrow L_j(y)}F_j(v)\le y$, so that
$\{v\,:\,F_j(v)\le y\}=(-\infty,L_j(y)]$. Consequently, for every
$(z,y)\in\mathbb R\times[0,1]$ with $L_j(y)\in\mathbb R$,
\[
F_j\left(\frac{x-z}{h}\right)\le y
\quad\Longleftrightarrow\quad
x-h\,L_j(y)-z\le 0 .
\]
Points $(z,y)$ with $L_j(y)=+\infty$ (resp.\ $L_j(y)=-\infty$) belong to
every (resp.\ no) element of $\mathcal I^{(j)}$ and therefore to no shattered
set; they may be discarded. On the remaining points, the membership functions
$(z,y)\mapsto x-h\,L_j(y)-z$, indexed by $(x,h)$, belong to the linear span
of the three functions $(z,y)\mapsto z$, $(z,y)\mapsto L_j(y)$ and
$(z,y)\mapsto 1$. By Theorem~13.9 of \cite{devroye2013probabilistic}, the VC dimension of
$\mathcal I^{(j)}$ is at most $3$, and Sauer's lemma yields
\[
\mathbb S_{\mathcal I^{(j)}}(n)\le(n+1)^3 .
\]

\noindent\emph{Bound on $\mathbb S_{\mathcal J^{(j)}}(n)$.}
Symmetrically, define
$R_j(u)=\inf\{v\in\mathbb R\,:\,G_j(v)\ge u\}\in[-\infty,+\infty]$, with the
convention $\inf\emptyset=+\infty$. Since $G_j$ is nondecreasing, the set
$\{v\,:\,G_j(v)\ge u\}$ is a right half-line; since $G_j$ is
right-continuous, if $R_j(u)\in\mathbb R$ then
$G_j(R_j(u))=\lim_{v\downarrow R_j(u)}G_j(v)\ge u$, so that
$\{v\,:\,G_j(v)\ge u\}=[R_j(u),+\infty)$. Consequently, for every
$(z,y)\in\mathbb R\times[0,1]$ with $R_j(1-y)\in\mathbb R$,
\[
G_j\left(\frac{x-z}{h}\right)\ge 1-y
\quad\Longleftrightarrow\quad
x-h\,R_j(1-y)-z\ge 0 .
\]
Points with $R_j(1-y)=\pm\infty$ are handled as before and discarded. The membership functions $(z,y)\mapsto x-h\,R_j(1-y)-z$ belong to the linear span of the three functions $(z,y)\mapsto z$, $(z,y)\mapsto R_j(1-y)$ and
$(z,y)\mapsto 1$, so that, by Theorem~13.9 of
\cite{devroye2013probabilistic} and Sauer's lemma again,
\[
\mathbb S_{\mathcal J^{(j)}}(n)\le(n+1)^3 .
\]

\noindent\emph{Conclusion.} Plugging these bounds into
\eqref{eq:product-bound-piecewise-h} gives
\[
\mathbb S_{\operatorname{sg}(\Psi)}(n)=\mathbb S_{\mathcal E}(n)
\le
(n+1)^{4q}\,(n+1)^{3q}\,(n+1)^{3q}
=(n+1)^{10q}.
\]
\qed

\subsection{Proof of Theorem \ref{th:density}}

We start by proving the second statement. Set $\psi_{h,x} (X) = K( { x-X} /{h} )$. Let $X$ be a random variable with density $f_X$. Noting that 
$$  h^{-1} \mathbb E \left[ \psi_{h,x} (X) \right]  = (f_X\ast K_h) (x) = f_h(x),$$
it follows that
$$ (nh) (\hat f_n(x) - f_h(x)) =    \sum_{i=1}^n \left( \psi_{h,x} (X_i)  - \mathbb E \left[ \psi_{h,x} (X) \right] \right) .$$
Let $ \Psi = \{\psi_{h,x} \,:\, x\in \mathbb R, \ h>0\}$.
Applying the first statement of Corollary \ref{theorem:cor2} with $ \Psi $ gives with probability at least $1- \delta$, for all $x\in \mathbb  R$ and $h>0$,
\begin{align*}
& \sum_{i=1}^n  \left( \psi_{h,x} (X_i)   - \mathbb E \left[  \psi_{h,x} (X_i) \right] \right)  \\
&\leq    \sqrt{4 n  \left(P ( \psi_{h,x} ) \wedge P_n ( \psi_{h,x} ) \right) \log\left( \frac{ \mathbb  S_{\mathrm{sg} (\Psi) }(2n)}{\delta}\right)  }  +   12  \log\left( \frac{ \mathbb  S_{\mathrm{sg} (\Psi)  }(2n)}{\delta}\right)   \\
 &\leq    \sqrt{4 (nh)  \left( f_h(x)  \wedge  \hat f_n(x) \right) \log\left( \frac{ \mathbb  S_{\mathrm{sg} (\Psi) }(2n)}{\delta}\right)  }  +   12  \log\left( \frac{ \mathbb  S_{\mathrm{sg} (\Psi )  }(2n)}{\delta}\right)   .
\end{align*}
Applying the second statement of Corollary \ref{theorem:cor2} with $ \Psi $ gives with probability at least $1- \delta$, for all $x\in \mathbb  R$ and $h>0$,
\begin{align*}
& - \sum_{i=1}^n  \left( \psi_{h,x} (X_i)   - \mathbb E \left[  \psi_{h,x} (X_i) \right] \right)   \\
 &\leq    \sqrt{4 (nh)  \left( f_h(x)  \wedge  \hat f_n(x) \right) \log\left( \frac{ \mathbb  S_{\mathrm{sg} (\Psi) }(2n)}{\delta}\right)  }  +   12  \log\left( \frac{ \mathbb  S_{\mathrm{sg} (\Psi )  }(2n)}{\delta}\right)   .
\end{align*}
An upper bound on $\mathbb{S}_{\operatorname{sg}(\Psi)}(2n)$ is now claimed invoking Lemma \ref{lemma:comp2}. We have 
$$
\frac{\mathbb S_{\operatorname{sg}(\Psi)}(2n)}{\delta}
\le
 \left( \frac{{3n} }{ {\delta} } \right)^{10q}.
$$
It follows, with probability at least $1-2\delta$, for all $x\in \mathbb  R$ and $h>0$,
\begin{align*}
& (nh) | \hat f_n - f_h (x) |\\
  &\leq   \sqrt{40q nh  \left( f_h(x)  \wedge  \hat f_n(x) \right)  \log\left( \frac{ 3n }{\delta}\right)  } +    120q  \log\left( \frac{ 3n}{\delta}\right) .
\end{align*}
The first statement follows from the same argument, invoking Lemma \ref{lemma:comp1} in place of Lemma \ref{lemma:comp2}, so that the exponent $10q$ is replaced by $6q$ and the constants $(40q, 120q)$ by $(24q,72q)$; in that case the supremum is over $x$ only, at fixed $h$.

\qed

\paragraph{Acknowledgement.} No specific funding was received for this research.

\bibliographystyle{chicago}
\bibliography{b2}

\begin{thebibliography}{}

\bibitem[\protect\citeauthoryear{Alexander}{Alexander}{1987}]{alexander1987central}
Alexander, K.~S. (1987).
\newblock The central limit theorem for empirical processes on
  vapnik--{\v{c}}ervonenkis classes.
\newblock {\em The Annals of Probability\/}, 178--203.

\bibitem[\protect\citeauthoryear{Anthony and Bartlett}{Anthony and
  Bartlett}{2009}]{anthony2009neural}
Anthony, M. and P.~L. Bartlett (2009).
\newblock {\em Neural network learning: Theoretical foundations}.
\newblock cambridge university press.

\bibitem[\protect\citeauthoryear{Anthony and Shawe-Taylor}{Anthony and
  Shawe-Taylor}{1993}]{anthony1993result}
Anthony, M. and J.~Shawe-Taylor (1993).
\newblock A result of {V}apnik with applications.
\newblock {\em Discrete Appl. Math.\/}~{\em 47\/}(3), 207--217.

\bibitem[\protect\citeauthoryear{Audibert, Munos, and Szepesv{\'a}ri}{Audibert
  et~al.}{2009}]{audibert2009exploration}
Audibert, J.-Y., R.~Munos, and C.~Szepesv{\'a}ri (2009).
\newblock Exploration--exploitation tradeoff using variance estimates in
  multi-armed bandits.
\newblock {\em Theoretical Computer Science\/}~{\em 410\/}(19), 1876--1902.

\bibitem[\protect\citeauthoryear{Baraud}{Baraud}{2016}]{baraud2016bounding}
Baraud, Y. (2016).
\newblock Bounding the expectation of the supremum of an empirical process over
  a (weak) vc-major class.
\newblock {\em Electronic Journal of Statistics\/}~{\em 10}, 1709--1728.

\bibitem[\protect\citeauthoryear{Bartlett and Lugosi}{Bartlett and
  Lugosi}{1999}]{bartlett1999inequality}
Bartlett, P. and G.~Lugosi (1999).
\newblock An inequality for uniform deviations of sample averages from their
  means.
\newblock {\em Statistics \& probability letters\/}~{\em 44\/}(1), 55--62.

\bibitem[\protect\citeauthoryear{Bartlett}{Bartlett}{1998}]{bartlett1998sample}
Bartlett, P.~L. (1998).
\newblock The sample complexity of pattern classification with neural networks:
  the size of the weights is more important than the size of the network.
\newblock {\em IEEE transactions on Information Theory\/}~{\em 44\/}(2),
  525--536.

\bibitem[\protect\citeauthoryear{Blumer, Ehrenfeucht, Haussler, and
  Warmuth}{Blumer et~al.}{1989}]{blumer1989learnability}
Blumer, A., A.~Ehrenfeucht, D.~Haussler, and M.~K. Warmuth (1989).
\newblock Learnability and the vapnik-chervonenkis dimension.
\newblock {\em Journal of the ACM (JACM)\/}~{\em 36\/}(4), 929--965.

\bibitem[\protect\citeauthoryear{Boucheron, Lugosi, and Massart}{Boucheron
  et~al.}{2013}]{boucheron2013concentration}
Boucheron, S., G.~Lugosi, and P.~Massart (2013).
\newblock {\em Concentration inequalities. A nonasymptotic theory of
  independence}.
\newblock Oxford University Press, Oxford.

\bibitem[\protect\citeauthoryear{Bousquet, Boucheron, and Lugosi}{Bousquet
  et~al.}{2003}]{bousquet2003introduction}
Bousquet, O., S.~Boucheron, and G.~Lugosi (2003).
\newblock Introduction to statistical learning theory.
\newblock In {\em Summer school on machine learning}, pp.\  169--207. Springer.

\bibitem[\protect\citeauthoryear{Chaudhuri and Dasgupta}{Chaudhuri and
  Dasgupta}{2010}]{chaudhuri2010rates}
Chaudhuri, K. and S.~Dasgupta (2010).
\newblock Rates of convergence for the cluster tree.
\newblock In {\em NeurIPS proceedings}, Volume~23, pp.\  343--351.

\bibitem[\protect\citeauthoryear{Cortes, Greenberg, and Mohri}{Cortes
  et~al.}{2019}]{cortes2019relative}
Cortes, C., S.~Greenberg, and M.~Mohri (2019).
\newblock Relative deviation learning bounds and generalization with unbounded
  loss functions.
\newblock {\em Annals of Mathematics and Artificial Intelligence\/}~{\em
  85\/}(1), 45--70.

\bibitem[\protect\citeauthoryear{Devroye, Gy{\"o}rfi, and Lugosi}{Devroye
  et~al.}{2013}]{devroye2013probabilistic}
Devroye, L., L.~Gy{\"o}rfi, and G.~Lugosi (2013).
\newblock {\em A probabilistic theory of pattern recognition}, Volume~31.
\newblock Springer Science \& Business Media.

\bibitem[\protect\citeauthoryear{Einmahl and Mason}{Einmahl and
  Mason}{2000}]{einmahl2000}
Einmahl, U. and D.~M. Mason (2000).
\newblock An empirical process approach to the uniform consistency of
  kernel-type function estimators.
\newblock {\em J. Funct. Anal.\/}~{\em 13\/}(1), 1--37.

\bibitem[\protect\citeauthoryear{Genovese and Wasserman}{Genovese and
  Wasserman}{2008}]{wasserman_2008}
Genovese, C. and L.~Wasserman (2008).
\newblock {Adaptive confidence bands}.
\newblock {\em The Annals of Statistics\/}~{\em 36\/}(2), 875 -- 905.

\bibitem[\protect\citeauthoryear{Gin{\'e} and Guillou}{Gin{\'e} and
  Guillou}{2002}]{gine+g:02}
Gin{\'e}, E. and A.~Guillou (2002).
\newblock Rates of strong uniform consistency for multivariate kernel density
  estimators.
\newblock {\em Ann. Inst. Henri Poincar\'{e} Probab. Stat.\/}~{\em 38\/}(6),
  907--921.
\newblock En l'honneur de J. Bretagnolle, D. Dacunha-Castelle, I. Ibragimov.

\bibitem[\protect\citeauthoryear{Gin\'{e} and Nickl}{Gin\'{e} and
  Nickl}{2009}]{gine2009}
Gin\'{e}, E. and R.~Nickl (2009).
\newblock An exponential inequality for the distribution function of the kernel
  density estimator, with applications to adaptive estimation.
\newblock {\em Probab. Theory Related Fields\/}~{\em 143\/}(3-4), 569--596.

\bibitem[\protect\citeauthoryear{Gin{\'e} and Nickl}{Gin{\'e} and
  Nickl}{2010a}]{gine_nivkl_wave}
Gin{\'e}, E. and R.~Nickl (2010a).
\newblock {Adaptive estimation of a distribution function and its density in
  sup-norm loss by wavelet and spline projections}.
\newblock {\em Bernoulli\/}~{\em 16\/}(4), 1137 -- 1163.

\bibitem[\protect\citeauthoryear{Gin{\'e} and Nickl}{Gin{\'e} and
  Nickl}{2010b}]{gine_nivkl_band}
Gin{\'e}, E. and R.~Nickl (2010b).
\newblock {Confidence bands in density estimation}.
\newblock {\em The Annals of Statistics\/}~{\em 38\/}(2), 1122 -- 1170.

\bibitem[\protect\citeauthoryear{Greenberg and Mohri}{Greenberg and
  Mohri}{2014}]{greenberg2014tight}
Greenberg, S. and M.~Mohri (2014).
\newblock Tight lower bound on the probability of a binomial exceeding its
  expectation.
\newblock {\em Statistics \& Probability Letters\/}~{\em 86}, 91--98.

\bibitem[\protect\citeauthoryear{Hagerup and R{\"u}b}{Hagerup and
  R{\"u}b}{1990}]{hagerup1990guided}
Hagerup, T. and C.~R{\"u}b (1990).
\newblock A guided tour of chernoff bounds.
\newblock {\em Information processing letters\/}~{\em 33\/}(6), 305--308.

\bibitem[\protect\citeauthoryear{Hsu, Kakade, and Zhang}{Hsu
  et~al.}{2012}]{hsu_12}
Hsu, D., S.~Kakade, and T.~Zhang (2012).
\newblock {A tail inequality for quadratic forms of subgaussian random
  vectors}.
\newblock {\em Electronic Communications in Probability\/}~{\em 17\/}(none), 1
  -- 6.

\bibitem[\protect\citeauthoryear{Juditsky and Lambert-Lacroix}{Juditsky and
  Lambert-Lacroix}{2003}]{juditsky2003nonparametric}
Juditsky, A. and S.~Lambert-Lacroix (2003).
\newblock Nonparametric confidence set estimation.
\newblock {\em Mathematical Methods of Statistics\/}~{\em 12\/}(4), 410--428.

\bibitem[\protect\citeauthoryear{Maurer and Pontil}{Maurer and
  Pontil}{2009}]{MaurerPontil2009}
Maurer, A. and M.~Pontil (2009).
\newblock Empirical bernstein bounds and sample-variance penalization.
\newblock In {\em Proceedings of the 22nd Annual Conference on Learning Theory
  (COLT 2009)}.

\bibitem[\protect\citeauthoryear{Nolan and Pollard}{Nolan and
  Pollard}{1987}]{nolan+p:1987}
Nolan, D. and D.~Pollard (1987).
\newblock {$U$}-processes: rates of convergence.
\newblock {\em Ann. Statist.\/}~{\em 15\/}(2), 780--799.

\bibitem[\protect\citeauthoryear{Pollard}{Pollard}{1982}]{pollard1982central}
Pollard, D. (1982).
\newblock A central limit theorem for empirical processes.
\newblock {\em Journal of the australian mathematical society\/}~{\em 33\/}(2),
  235--248.

\bibitem[\protect\citeauthoryear{Pollard}{Pollard}{1995}]{pollard1995uniform}
Pollard, D. (1995).
\newblock Uniform ratio limit theorems for empirical processes.
\newblock {\em Scandinavian Journal of Statistics\/}, 271--278.

\bibitem[\protect\citeauthoryear{Portier}{Portier}{2025}]{portier2025nearest}
Portier, F. (2025).
\newblock Nearest neighbor empirical processes.
\newblock {\em Bernoulli\/}~{\em 31\/}(1), 312--332.

\bibitem[\protect\citeauthoryear{Robins and van~der Vaart}{Robins and van~der
  Vaart}{2006}]{robins2006adaptive}
Robins, J. and A.~van~der Vaart (2006).
\newblock Adaptive nonparametric confidence sets.
\newblock {\em Ann. Statist.\/}~{\em 34\/}(1), 229--253.

\bibitem[\protect\citeauthoryear{Rockafellar}{Rockafellar}{1997}]{rockafellar1997convex}
Rockafellar, R.~T. (1997).
\newblock {\em Convex analysis}, Volume~28.
\newblock Princeton university press.

\bibitem[\protect\citeauthoryear{Van Der~Vaart and Wellner}{Van Der~Vaart and
  Wellner}{1996}]{wellner1996}
Van Der~Vaart, A.~W. and J.~A. Wellner (1996).
\newblock {\em Weak Convergence and Empirical Processes. With Applications to
  Statistics}.
\newblock Springer Series in Statistics. New York: Springer-Verlag.

\bibitem[\protect\citeauthoryear{Vapnik}{Vapnik}{2013}]{vapnik2013nature}
Vapnik, V. (2013).
\newblock {\em The nature of statistical learning theory}.
\newblock Springer science \& business media.

\bibitem[\protect\citeauthoryear{Vapnik and Chervonenkis}{Vapnik and
  Chervonenkis}{2015}]{vapnik2015uniform}
Vapnik, V.~N. and A.~Y. Chervonenkis (1971 - repreinted in 2015).
\newblock On the uniform convergence of relative frequencies of events to their
  probabilities.
\newblock In {\em Measures of complexity}, pp.\  11--30. Springer, Cham.
\newblock Reprint of Theor. Probability Appl. {16} (1971), 264--280.

\bibitem[\protect\citeauthoryear{Xue and Kpotufe}{Xue and
  Kpotufe}{2018}]{xue2018achieving}
Xue, L. and S.~Kpotufe (2018).
\newblock Achieving the time of 1-nn, but the accuracy of k-nn.
\newblock In {\em International Conference on Artificial Intelligence and
  Statistics}, pp.\  1628--1636. PMLR.

\end{thebibliography}

\end{document}